\documentclass[10pt,letterpaper, boxed, cm]{amsart}

\usepackage[margin=1in]{geometry}
\usepackage{amsmath,amsthm,amsfonts,amssymb,amscd}
\usepackage{lastpage}
\usepackage{fancyhdr}
\usepackage{mathrsfs}
\usepackage{xcolor}
\usepackage{graphicx}
\usepackage{listings}
\usepackage{hyperref}
\usepackage{enumitem}
\usepackage{relsize}
\usepackage{tikz-cd}
\usepackage{mdwlist}
\usepackage{multicol}
\usepackage{stmaryrd}
\usepackage{float}
\usepackage{adjustbox}
\usepackage{tikz}
\usetikzlibrary{matrix, calc, arrows}

\hypersetup{
    colorlinks=true, %set true if you want colored links
    linktoc=all,     %set to all if you want both sections and subsections linked
    linkcolor=blue,  %choose some color if you want links to stand out
}

\setlist[enumerate]{label=(\alph*)}
\usetikzlibrary{graphs,decorations.pathmorphing,decorations.markings}
\tikzcdset{scale cd/.style={every label/.append style={scale=#1},
    cells={nodes={scale=#1}}}}

\newcommand{\Z}{\mathbb{Z}}

\newcommand{\N}{\mathbb{N}}

\newcommand{\Hom}{\operatorname{Hom}}

\newcommand{\tr}{\operatorname{tr}}
\newcommand{\rk}{\operatorname{rk}}
\newcommand{\inv}{^{-1}}

\newcommand{\End}{\operatorname{End}}

\newcommand{\Res}{\operatorname{Res}}
\newcommand{\Ind}{\operatorname{Ind}}
\newcommand{\Inf}{\operatorname{Inf}}

\newcommand{\Iso}{\operatorname{Iso}}
\newcommand{\id}{\operatorname{id}}

\newcommand{\im}{\operatorname{im}}

\newcommand{\calA}{\mathcal{A}}
\newcommand{\calB}{\mathcal{B}}
\newcommand{\calC}{\mathcal{C}}

\newcommand{\calE}{\mathcal{E}}
\newcommand{\calF}{\mathcal{F}}

\newcommand{\calH}{\mathcal{H}}

\newcommand{\calX}{\mathcal{X}}
\newcommand{\calY}{\mathcal{Y}}

\newcommand{\Syl}{\operatorname{Syl}}

\newcommand{\triv}{\mathbf{triv}}
\newcommand{\coker}{\operatorname{coker}}
\newcommand{\catmod}{\mathbf{mod}}

\newcommand{\stmod}{\mathbf{stmod}}
\newcommand{\sttriv}{\mathbf{sttriv}}
\newcommand{\stCh}{\mathbf{stCh}}
\newcommand{\stC}{\mathbf{st}\calC}
\newcommand{\stK}{\mathbf{stK}}

\newtheorem{theorem}{Theorem}[section]
\newtheorem{lemma}[theorem]{Lemma}
\newtheorem{prop}[theorem]{Proposition}
\newtheorem{corollary}[theorem]{Corollary}

\newtheorem*{theorem*}{Theorem}
\newtheorem{conjecture}[theorem]{Conjecture}

\theoremstyle{remark}
\newtheorem{remark}[theorem]{Remark}
\newtheorem{example}[theorem]{Example}

\theoremstyle{definition}
\newtheorem{definition}[theorem]{Definition}

\newtheorem{construction}[theorem]{Construction}

\begin{document}

    \title{Relatively endotrivial complexes}
    \author{Sam K. Miller}
    %\contrib[...]{...} %optional2
    \address{University of California, Santa Cruz, Department of Mathematics} %required
    %\curraddr{...} %optional
    \email{sakmille@ucsc.edu} %optional
    %\urladdr{...} %optional
    %\dedicatory{...} %optional
    %\date{...} %—3
    %\thanks{...} %optional
    %\translator{...} %—3
    \subjclass[2010]{20J05, 20C05, 20C20} %required
    \keywords{Endotrivial complex, relative projectivity, $p$-permutation, restriction, induction, endosplit $p$-permutation resolution} %optional

    \maketitle

    \begin{abstract}
        Let $G$ be a finite group and $k$ be a field of characteristic $p > 0$. In prior work, we studied endotrivial complexes, the invertible objects of the bounded homotopy category of $p$-permutation $kG$-modules $K^b({}_{kG}\triv)$. Using the notion of projectivity relative to a $kG$-module, we expand on this study by defining notions of ``relatively'' endotrivial chain complexes, analogous to Lassueur's construction of relatively endotrivial $kG$-modules. We obtain equivalent characterizations of relative endotriviality and find corresponding local homological data which almost completely determine the isomorphism class of a relatively endotrivial complex. We show this local data must partially satisfy the Borel-Smith conditions, and consider the behavior of restriction to subgroups containing Sylow $p$-subgroups $S$ of $G$.
    \end{abstract}

    \section{Introduction}

    In her doctoral dissertation \cite{CL12}, Lassueur introduced the notion of relatively endotrivial $kG$-modules. This built on the notion of relative projectivity with respect to modules first introduced by Okuyama in \cite{O} and later studied by Carlson and coauthors in \cite{CP96} and others. One important consequence of Lassueur's work was a new generalization of the Dade group of a $p$-group to all finite groups. This construction coincided with a previous generalization given by Linckelmann and Mazza in \cite{LiMa09}, who instead considered the Dade group of a fusion system.

    Similarly to how endotrivial modules are the invertible objects of the stable module category, ${}_{kG}\stmod$, endotrivial chain complexes, which we first considered in \cite{SKM23}, are the invertible objects of the homotopy category $K^b({}_{kG}\triv)$ of $p$-permutation modules. Both of these categories are tensor-triangulated, and classifying the invertible objects of such categories is a pertinent question. Furthermore, endotrivial complexes induce splendid Rickard autoequivalences, derived equivalences which induce ``local'' derived equivalences and which are a key part of Brou\'{e}'s abelian defect group conjecture. Although endotrivial complexes have a name similar to that of endotrivial modules, the theory of these complexes is markedly different. We refer the reader to \cite[Section 6]{SKM23} where we explicitly classify all endotrivial complexes for some classes of $p$-groups, and \cite[Section 3.3]{Ma19} for an overview of the classification of endotrivial modules for $p$-groups, completed by Carlson and Th\'{e}venaz.

    In this paper, we apply the ideas present in Lassueur's work to the study of endotrivial complexes, which the author first studied in \cite{SKM23}. Our goal is to broadly develop analogous theory for relatively endotrivial complexes, and to use this machinery to better understand endotrivial complexes. In Lassueur's work, there is one clear definition choice for a relatively endotrivial module, but in the endotrivial complex setting, there is more ambiguity. What the notion of a ``relatively endotrivial chain complex'' should be has multiple reasonable definitions, so we provide multiple definitions of increasing refinement. The following is Definition \ref{endotrivialdef}.

    \begin{definition}
        Let $C \in Ch^b({}_{kG}\triv)$ and let $V$ be a $kG$-module.
        \begin{enumerate}
            \item $C$ is \textit{weakly $V$-endotrivial} if $\End_k(C) \cong C^* \otimes_k C \simeq k[0] \oplus D$, where $D$ is a bounded chain complex of $V$-projective $kG$-modules.
            \item $C$ is \textit{strongly $V$-endotrivial} if $\End_k(C) \cong C^* \otimes_k C \simeq k[0] \oplus D$, where $D$ is a bounded $V$-projective $kG$-chain complex.
            \item $C$ is \textit{$V$-endosplit-trivial} if $\End_k(C) \cong C^* \otimes_k C \simeq k[0] \oplus M[0]$, where $M$ is a $V$-projective $kG$-module.
        \end{enumerate}
        Here, $\simeq$ denotes homotopy equivalence.
    \end{definition}

    We allow $V$ to be the zero module, in which case we recover the definition of an endotrivial complex. Weakly and strongly $V$-endotrivial complexes are the invertible objects of appropriate stable homotopy categories of chain complexes of $p$-permutation $kG$-modules. $V$-endosplit-trivial complexes are well-behaved examples of strongly $V$-endotrivial complexes and are examples of endosplit $p$-permutation resolutions, as defined by Rickard. Analogous to how endotrivial complexes induce splendid Rickard complexes, when $V$ is a projective $kG$-module, weakly $V$-endotrivial complexes induce splendid stable equivalences.

    Many of the techniques used to study endotrivial complexes in \cite{SKM23} may be performed in the relative setting. In \cite{SKM23}, we provided a $p$-local equivalent condition that determines when a chain complex $C \in Ch^b({}_{kG}\triv)$ is endotrivial. Similar characterizations are provided here for weak $V$-endotriviality and $V$-endosplit triviality. Furthermore, we show that for these cases, one may assume that $V$ is a $p$-permutation $kG$-module. As a result, the only types of relative endotriviality arise from projectivity relative to a family of subgroups. The detection theorems are as follows; these are Theorem \ref{equivdefinitionweakendotrivial} and Theorem \ref{cor:vesplittrivaltdefs}. For these statements, we say $kG$-module $V$ is \textit{absolutely $p$-divisible} if all of its indecomposable summands have $k$-dimension divisible by $p$; this condition is necessary for $V$-projectivity to be well-behaved.

    \begin{theorem}
        Let $C \in Ch^b({}_{kG}\triv)$, let $V$ be a $p$-permutation $kG$-module which is absolutely $p$-divisible, and let $\calX_V$ be the set of $p$-subgroups $P$ of $G$ for which $V(P) = 0$.
        \begin{enumerate}
            \item  $C$ is weakly $V$-endotrivial if and only if for all $P \in\calX_V$, $C(P)$ has nonzero homology concentrated in exactly one degree, with the nontrivial homology having $k$-dimension one.

            \item The following are equivalent.
            \begin{enumerate}
                \item $C$ is $V$-endosplit-trivial.
                \item $C(P)$ has nonzero homology concentrated in exactly one degree for all $P \in s_p(G)$, and that homology has $k$-dimension one when $P \in \calX_V$.
                \item $C(P)$ has nonzero homology concentrated in exactly one degree for all $P \in s_p(G)$, and for the unique $i\in \Z$ satisfying $H_i(C) \neq 0$, $H_i(C)$ is a $V$-endotrivial $kG$-module.
            \end{enumerate}
        \end{enumerate}

    \end{theorem}

    When studying endotrivial modules or chain complexes, one constructs a corresponding group parameterizing the objects, with group law induced by $\otimes_k$. This is the \textit{Picard group} of the corresponding category. We carry out analogous constructions here, with the goal of determining the isomorphism class and generators of our groups. One of the first landmark results in the study of endotrivial modules, due to Puig in \cite{Pu90}, asserts that when $G$ is a $p$-group, the corresponding group, $T_k(G)$, is finitely generated abelian. We show an analogous result, showing that the group of weakly $V$-endotrivial complexes and the group of $V$-endosplit-trivial complexes are both finitely generated abelian. To do this, we use a similar technique as in \cite{SKM23}; for each relatively endotrivial complex, we associate a superclass function called the \textit{h-marks} of the complex. In Section 9, we show that the association of a relatively endotrivial complex $C$ to its h-mark function $h_C$ induces a group homomorphism that proves finite generation of the corresponding group. Additionally, we show that the h-marks must satisfy numerical conditions known as the Borel-Smith conditions.

    A question of interest in the study of (relatively) endotrivial modules is the image of the homomorphism induced by restriction to a Sylow $p$-subgroup or its normalizer. We consider an analogous question for relatively endotrivial complexes. One crucial lemma we prove along the way is a Mackey-like formula describing how the Brauer construction and induction commute. As a result, the Green correspondence for certain $G$-stable chain complexes preserves relative endotriviality. In particular, we can describe the group $\calE_k(G)$ of endotrivial complexes as a direct product of the $G$-stable subgroup $\calE_k(S)^G \leq \calE_k(S)$, where $S$ is a Sylow $p$-subgroup of $G$, and the torsion subgroup of $\calE_k(G)$, which is canonically isomorphic to $\Hom(G,k^\times)$. The following result arises from Theorem \ref{endotrivialrestriction} and reduces the question of classifying all endotrivial complexes to the case of $p$-groups.

    \begin{theorem}
        Let $S \in \Syl_p(G)$. Then the group homomorphism $\Res^G_S: \calE_{k}(G) \to \calE_k(S)^G$ is surjective. Moreover, we have a split exact sequence of abelian groups \[0 \to \Hom(G,k^\times) \to \calE_k(G) \xrightarrow{\Res^G_S} \calE_k(S)^G \to 0,\] with retraction $\calE_k(G) \to \Hom(G,k^\times)$ given by $[C] \mapsto H_{h_C(1)}(C) \in \Hom(G,k^\times)$, regarding the $k$-dimension one $kG$-module $H_{h_C(1)}(C)$ as a degree one Brauer character.
    \end{theorem}

    Finally, some of the machinery developed in this paper can be applied towards understanding endosplit $p$-permutation resolutions. These resolutions, first introduced by Rickard, were used to resolve Brou\'{e}'s abelian defect group conjecture for blocks of $p$-nilpotent groups in \cite{R96} and $p$-solvable groups in \cite{HL00}. The following results are Proposition \ref{prop:equivendosplit} and Corollary \ref{fusionstableinduction}(c).

    \begin{theorem}
        \begin{enumerate}
            \item Let $C$ be a bounded chain complex of $p$-permutation $kG$-modules. The following are equivalent:
            \begin{itemize}
                \item $C$ is a shifted endosplit $p$-permutation resolution.
                \item For every $P \in s_p(G)$, one of the following holds: either $C(P)$ is contractible or the nonzero homology of $C(P)$ is concentrated in exactly one degree.
            \end{itemize}

            \item Let $H \leq G$ and let $N$ be a $kH$-module with an endosplit $p$-permutation resolution $C$. $\Ind^G_H C$ is an endosplit $p$-permutation resolution if and only if $C$ is $G$-stable, that is, for every pair of $G$-conjugate $p$-subgroups $P,Q \in s_p(H)$ for which both $C(P)$ and $C(Q)$ are noncontractible, $C(P)$ and $C(Q)$ have nonzero homology concentrated in the same unique degree. If this occurs, $\Ind^G_H C$ is an endosplit $p$-permutation resolution for $\Ind^G_H N$.
        \end{enumerate}

    \end{theorem}

    The paper is organized as follows. Sections 2 through 5 serve primarily as preliminary information. In Section 2, we review relative projectivity for chain complexes and provide a statement of the Green correspondence for chain complexes, which was introduced in \cite{CWZ20}. Section 3 mostly follows \cite{CL11} in establishing properties of relative projectivity with respect to modules, extending some statements to relatively projective chain complexes as well. Section 4 recalls relatively stable module and chain complex categories and describes how they behave in some detail. Section 5 recalls a few necessary properties about relatively endotrivial modules.

    Sections 6 and onwards contain most of the main results of the paper. Section 6 introduces the various definitions of endotrivial complexes, their corresponding groups, and some basic properties. Section 7 is dedicated to proving the equivalent local characterization of weakly $V$-endotrivial complexes, and Section 8 is dedicated to proving the equivalent local characterization for $V$-endosplit-trivial complexes. Section 9 introduces h-mark homomorphisms, analogous to the h-mark homomorphism given in \cite{SKM23}, and proves that the group of weakly $V$-endotrivial complexes and the group of $V$-endosplit-trivial complexes are finitely generated. Section 10 introduces Borel-Smith functions and shows that h-marks must partially satisfy the Borel-Smith conditions. Section 11 analyzes the case of when $V$ is projective, and Section 12 studies restriction to subgroups containing Sylow $p$-subgroups.

    \textbf{Notation:} $G$ is always a finite group, $p$ a prime, and $k$ a field of characteristic $p$ (not necessarily large enough for $G$). All $kG$-modules are assumed left and finitely generated. We write ${}_{kG}\catmod$ to denote the category of finitely generated $kG$-modules, and ${}_{kG}\triv$ to denote the full subcategory consisting of $p$-permutation modules. Given a preadditive category $\calA$, $Ch^b(\calA)$ denotes the bounded category of $\calA$-chain complexes and $K^b(\calA)$ denotes the bounded homotopy category of $\calA$. ${}_{kG}\stmod$ denotes the stable module category of ${}_{kG}\catmod.$  If $M,N \in \calA$, we write $M \mid N$ to denote that $M$ is isomorphic to a direct summand of $N$. The symbol $\otimes_k$ denotes the tensor product of $kG$-modules or chain complexes of $kG$-modules with diagonal $G$-action. $(-)^*$ denotes the $k$-dual $\Hom_k(-,k)$ of $kG$-modules (resp. chain complexes of $kG$-modules) which is again a $kG$-module (resp. chain complex of $kG$-modules). $s_p(G)$ denotes the set of $p$-subgroups of $G$ and $\Syl_p(G)$ denotes the set of Sylow $p$-subgroups of $G$. If $M \in {}_{kG}\catmod$, $M[i]$ denotes the chain complex with $M$ in degree $i$ and $0$ elsewhere. Given any $p$-subgroup $P \leq G$, the Brauer construction at $P$ is written as a functor $-(P): {}_{kG}\catmod \to {}_{k[N_G(P)/P]}\catmod$.

    \section{Projectivity relative to subgroups for chain complexes}

    We begin by reviewing the notions of relative projectivity, vertices, sources, and the Green correspondence for chain complexes of $kG$-modules, building up to a chain complex theoretic Green correspondence introduced by Carlson, Wang, and Zhang in \cite{CWZ20}. The structural results here mirror those for $kG$-modules. We assume familiarity with the module-theoretic versions of these concepts.

    \begin{definition}

        A chain complex $C$ of $kG$-modules is $H$-projective if there exists some $kH$-complex $D$ for which $C$ is isomorphic to a direct summand of $\Ind^G_H D$. If $\calX$ is a collection of subgroups of $G$, $C$ is $\calX$-projective if, for any indecomposable direct summand $D$ of $C$, there exists an $H \in \calX$ for which $D$ is $H$-projective.

        Given a set of subgroups $\calX$ of $G$, we set $Ch^b({}_{kG}\catmod,\calX)$ to be the category of $\calX$-projective $kG$-chain complexes. If $\calX = \{H\}$, we write $Ch^b({}_{kG}\catmod,H)$.

    \end{definition}

    \begin{remark}
        \begin{enumerate}
            \item Recall the categories $Ch^b({}_{kG}\catmod)$ and $K^b({}_{kG}\catmod)$ satisfy the Krull-Schmidt property. That is, if $X \cong A_1 \oplus \cdots \oplus A_k = B_1 \oplus \cdots \oplus B_l$, are two decompositions of an object $X$ in one of the aforementioned categories into indecomposable direct summands, then $k= l$ and there exists some $\sigma \in S_k$ for which $A_i \cong B_{\sigma(i)}$ for all $i \in \{1, \dots, k\}$.

            \item There is a subtle difference between $\calX$-projectivity for chain complexes and $\calX$-projectivity for modules. In the module-theoretic case, a $kG$-module is projective if and only if it is $1$-projective and projective as $k$-module. The analogous statement does not hold for chain complexes. For instance, the projective objects in the chain complex category $Ch^b({}_{kG}\catmod)$ are split acyclic chain complexes of projective $kG$-modules. In particular, they are additively generated by contractible chain complexes of the form $0\to P \xrightarrow{\sim} P \to 0$ for projective $kG$-modules $P$. On the other hand, the singleton chain complex $kG[0]$ is a $1$-projective chain complex, but is not a projective object since it is not acyclic.

            Furthermore, a $\calX$-projective chain complex is not the same thing as a chain complex of $\calX$-projective modules. For example, let $k$ have characteristic 2, let $G = C_2 = \langle \sigma \rangle$, and define the endomorphism $f \in \End_{kC_2}(kC_2)$ by $k$-linearizing the assignment $1 \mapsto 1 + \sigma$. Then the chain complex $0 \to kC_2 \xrightarrow{f} kC_2 \to 0$ is not $1$-projective as a chain complex of $kC_2$-modules, despite having $1$-projective objects in each degree, since the homology of $C$ is isomorphic to the trivial $kC_2$-module in both degrees. On the other hand, if $M$ is a $\calX$-projective $kG$-module, then $M[i]$ is a $\calX$-projective chain complex, where $M[i]$ denotes the chain complex with $M$ in degree $i$ and zero in all other degrees.

            \item We briefly recall that a natural isomorphism of modules induces a natural isomorphism of chain complexes as follows. Let $\calA, \calB$ be additive categories. If $F,G: \calA \to \calB$ are naturally isomorphic, then $F,G$ induce functors of chain complex categories, $F,G: Ch^b(\calA) \to Ch^b(\calB)$, by applying $F,G$ componentwise. Taking chain maps induced componentwise from the natural isomorphism $F \cong G$ induces another natural isomorphism $F,G: Ch^b(\calA) \to Ch^b(\calB)$.

            Moreover, a natural isomorphism of bifunctors extends to a natural isomorphism of bicomplexes. In particular, natural isomorphisms involving internal homs or tensor products may be extended to their chain complex counterparts, as these constructions arise from taking total complexes of bicomplexes. Moreover, the Mackey formula is a natural isomorphism, as are all standard adjunctions involving induction and restriction, therefore they extend to natural isomorphisms of complexes as well.
        \end{enumerate}

    \end{remark}

    In general, projectivity relative to a subgroup for chain complexes behaves similarly to those in the module-theoretic setting, and one may verify that many of the proofs of standard facts generalize to chain complexes with minimal changes. In addition, we have a notion of vertices and sources for chain complexes.

    \begin{definition}
        An indecomposable chain complex $C$ has \textit{vertex $P$} if $P$ is minimal with respect to the property that $C$ is $P$-projective. If $D$ is an indecomposable complex for which $C \mid \Ind^G_P D$, we say $(P,D)$ is a \textit{vertex-source} pair for $C$.
    \end{definition}

    As in the module case, the set of vertices is always a conjugacy classes of $p$-subgroups.

    \begin{prop}
        Let $C$ be a chain complex of $kG$-modules. The set of vertices of $C$ forms a single $G$-conjugacy class of $p$-subgroups.
    \end{prop}
    \begin{proof}
        This is \cite[Chapter 4, Section 3, Theorem 3.3]{NT89}, and the proof adapts to chain complexes.
    \end{proof}

    \begin{theorem}{(Classical Green correspondence for chain complexes)}\label{sylowgreencorforcomplexes}
        Let $S$ be a Sylow $p$-subgroup of $G$, and let $H$ be a subgroup of $G$ with $N_G(S) \leq H \leq G$.
        \begin{enumerate}
            \item Given a chain complex of $kG$-modules $C$ with vertex $S$, $\Res^G_H C$ has a unique direct summand $f(C)$ with vertex $S$. All other summands of $\Res^G_H C$ have vertices that are $G$-conjugate to proper subgroups of $S$.
            \item Given a chain complex of $kH$-modules $D$ with vertex $S$, $\Ind^G_H D$ has a unique direct summand $g(D)$ with vertex $S$. All other summands of $\Ind^G_H D$ have vertices that are $G$-conjugate to proper subgroups of $S$.
            \item $f$ and $g$ are inverses of each other.
        \end{enumerate}
    \end{theorem}
    \begin{proof}
        This follows from the more general Green correspondence for chain complexes, \cite[Theorem 8.1]{CWZ20}, by choosing $H$ as before and $\mathfrak{B} = \{S\}$, again similar to the module-theoretic case.
    \end{proof}

    \section{Relative projectivity for modules and chain complexes with respect to modules}

    We assume familiarity with $p$-permutation modules, trivial source modules (indecomposable $p$-permutation modules), and the Brauer construction. We refer the reader to \cite{L181} or \cite[Section 2]{SKM23} for a review of these topics.

    \textbf{Notation:} For this section, let ${}_{kG}\calC$ denote either the category ${}_{kG}\catmod$ or the category $Ch^b({}_{kG}\catmod)$. Both are additive, $k$-linear, symmetric monoidal categories that satisfy the Krull-Schmidt property on objects. These categories come equipped with induction, restriction, and inflation functors, as well as the Brauer construction. If $M \in Ch^b({}_{kG}\catmod)$ and $V \in {}_{kG}\catmod$, $M \otimes_k V$ denotes the tensor product of chain complexes $M \otimes_k V[0]$.

    Parts of the following section is adapted from \cite[Section 2]{CL11}, and we extend the previous results to the chain complex category with minimal work. Here, we review a generalized notion of relative projectivity, first introduced by Okuyama in \cite{O} and later studied by Carlson and Peng in \cite{CP96}.

    \begin{definition}
        Let $V$ be a $kG$-module.
        \begin{enumerate}
            \item $M \in {}_{kG}\calC$ is \textit{$V$-projective} or \textit{projective relative to $V$} if there exists an object $ N \in {}_{kG}\calC$ such that $M$ is isomorphic to a direct summand of $V \otimes_k N$.
            \item A short exact sequence in ${}_{kG}\calC$ $0 \to A \xrightarrow{\alpha} B \xrightarrow{\beta} C \to 0$ is \textit{$V$-split} if the short exact sequence $0 \to V \otimes_k A \xrightarrow{\id \otimes \alpha} V \otimes_k B \xrightarrow{\id \otimes \beta} V \otimes_k C \to 0$ in ${}_{kG}\calC$ splits.
        \end{enumerate}
    \end{definition}

    \begin{remark}
        The notion of a $V$-injective $kG$-module can be defined dually, and since $kG$ is a symmetric algebra, the classes of $V$-projective objects and $V$-injective objects coincide.
    \end{remark}

    \begin{definition}
        We write ${}_{kG}\catmod(V)$ for the full subcategory of $V$-projective modules in ${}_{kG}\catmod$. Similarly, we write $Ch^b({}_{kG}\catmod,V)$ for the full subcategory of $V$-projective chain complexes in $Ch^b({}_{kG}\catmod)$. More generally, we write ${}_{kG}\calC(V)$ to denote the corresponding full subcategory of $V$-projective objects in ${}_{kG}\calC$.

        We say a module $U$ is a \textit{generator} for ${}_{kG}\calC(V)$ if and only if ${}_{kG}\calC(U) = {}_{kG}\calC(V)$.

        Similarly, if $H \leq G$, we write ${}_{kG}\calC(H)$ for the subcategory of ${}_{kG}\calC$ with objects consisting of $H$-projective objects, and if $\calX$ is a set of subgroups of $G$ closed under conjugacy, write ${}_{kG}\calC(\calX)$ for the subcategory of ${}_{kG}\calC$ with objects consisting of $\calX$-projective objects.

    \end{definition}

    We refer the reader to \cite[Section 2]{CL11} and \cite[Section 2]{CL12} for an extensive list of properties regarding $V$-projectivity. It is a tedious but easy exercise to verify that most of the properties given also apply for chain complexes of $kG$-modules. We will cite and apply some of these results, but for chain complexes rather than modules, without further comment.

    \begin{remark}\label{rmk:sameasprojectivityreltosubgroups}
        The notion of $V$-projectivity encompasses projectivity relative to a subgroup, and more generally, the notion of projectivity relative to a family of subgroups. Using Frobenius reciprocity, it is easy to show that projectivity relative to $H$ is equivalent to projectivity relative to the $kG$-module $k[G/H] = \Ind^G_H k$, and that a short exact sequence is $H$-split if and only if it is $k[G/H]$-split. Similarly, projectivity relative to $\calX$ is equivalent to projectivity relative to $\bigoplus_{H \in \calX} k[G/H]$ and that a short exact sequence is $\calX$-split if and only if it is $\bigoplus_{H \in \calX} k[G/H]$-split if and only if it splits upon restriction to \textit{every} $H \in \calX$.
    \end{remark}

    $V$-projectivity is ``bounded below'' by the vertices that occur in all indecomposable summands of $V$.

    \begin{prop}\label{prop:3.5}
        If $V$ is an indecomposable $kG$-module with vertex $P$, then ${}_{kG}\calC(V) \subseteq {}_{kG}\calC(P)$.
    \end{prop}
    \begin{proof}
        Suppose $V \mid \Ind^G_P X$ for some $kG$-module $X$. Then given any $V$-projective $kG$-module $M$, we have \[M \mid V\otimes_k N \mid (\Ind^G_P X) \otimes_k N \cong \Ind^G_P(X \otimes_k \Res^G_P N).\] Thus $M$ is $P$-projective.
    \end{proof}

    \begin{prop}\label{indrespreservevproj}
        Let $H \leq G$, $P\in s_p(G)$, and $N \trianglelefteq G$.
        \begin{enumerate}
            \item $\Res^G_H {}_{kG}\calC(V) \subseteq {}_{kH}\calC(\Res^G_H V)$.
            \item $\Ind^G_H {}_{kH}\calC(V) \subseteq {}_{kG}\calC(\Ind^G_H V)$.
            \item $\Inf^G_{G/N} {}_{k[G/N]}\calC(V) \subseteq {}_{kG}\calC(\Inf^G_{G/N} V)$.
            \item Let $\phi: \tilde{G}\to G$ be a group homomorphism and let $M$ be a $V$-projective $kG$-module. $\Res_\phi {}_{kG}\calC(V) \subseteq {}_{k\tilde{G}}\calC(\Res_\phi V)$, where $\Res_\phi$ denotes restriction along the induced algebra homomorphism $k\tilde{G} \to kG$.
            \item If $V$ is a $p$-permutation $kG$-module, then $\big({}_{kG}\calC(V)\big)(P)\subseteq {}_{k[N_G(P)/P]}\calC(V(P))$.
        \end{enumerate}
    \end{prop}
    \begin{proof}
        All of these are proven in \cite[Lemma 2.1.1]{CL11} for modules except (e), and each of the proofs extends to the chain complex case.

        For (e), suppose $M \mid V\otimes_k N$ for some $N \in {}_{kG}\calC$. Then $M(P) \mid (V\otimes_k N)(P) \cong V(P)\otimes_k N(P)$ in ${}_{k[N_G(P)/P]}\calC$ by \cite[Theorem 5.8.10]{L181}, since $V$ is $p$-permutation.
    \end{proof}

    If $V$ is a $p$-permutation module, then $V$-projectivity is the same as projectivity relative to a family of subgroups, as we now demonstrate.

    \begin{theorem}\label{cor:ppermrelprojissubgroup}
        \begin{enumerate}
            \item Let $S \in \Syl_p(G)$. Let $V$ and $W$ be two $kG$-modules. Then ${}_{kG}\calC(V) = {}_{kG}\calC(W)$ if and only if ${}_{kS}\calC(\Res^G_S V) = {}_{kS}\calC(\Res^G_S W)$.

            \item If $V$ is a trivial source module with vertex $P$, then ${}_{kG}\calC(V) = {}_{kG}\calC(k[G/P]) = {}_{kG}\calC(P)$.

            \item If $V$ is a $p$-permutation module, let $\calX_V$ be the set of $p$-subgroups $P$ of $G$ for which $V(P) \neq 0$. Then ${}_{kG}\calC(V) = {}_{kG}\calC(\calX_V)$.
        \end{enumerate}
    \end{theorem}
    \begin{proof}
        The proof of (a) is given in \cite[Corollary 2.1.4]{CL11}, and extends to chain complexes.

        For (b), observe that if $V$ is a trivial source module with vertex $P$, then $V \mid k[G/P]$. We have \[\Res^G_S k[G/P] \cong \bigoplus_{x \in [S\backslash G/P]} k[S/S\cap {}^xP].\] This decomposition is the complete decomposition of $\Res^G_S k[G/P]$ into indecomposables, so $\Res^G_S V$ must be isomorphic to a direct sum of modules of the form $k[S/Q]$ for some subgroup $Q$ of $G$ with $Q \leq_G P$. If there does not exist a direct summand of $\Res^G_S V$ isomorphic to $k[S/P]$, then $V(P) = 0$, implying that $V$ has vertex properly contained in a conjugate of $P$, a contradiction. Thus $\Res^G_S V$ has a direct summand isomorphic to $k[S/P]$, and from the direct sum decomposition, all other indecomposable direct summands of $\Res^G_S M$ have a vertex contained in $P$.

        Now, if $\calX$ is the poset of all subgroups $Q$ of $S$ for which $\Res^G_S M$ has a summand isomorphic to $k[S/Q]$, then ${}_{kS}\calC(\Res^G_S M) = {}_{kS}\calC(\calX)$ by Remark \ref{rmk:sameasprojectivityreltosubgroups}. Note that if $R \leq Q \leq S,$ then ${}_{kS}\calC(R) \subseteq {}_{kS}\calC(Q)$. By the previous paragraph's argument, $\calX$ has as its maximal elements all $S$-conjugates of $P$, therefore ${}_{kS}\calC(\calX) = {}_{kS}\calC(P)$. A similar argument produces that ${}_{kS}\calC(\Res^G_S k[G/P]) = {}_{kS}\calC(P)$, and the result follows by (a).

        Finally, (c) follows from Proposition \cite[Proposition 2.2.2]{CL12} and the fact that since $V$ is $p$-permutation, $V(P) = 0$ if and only if $P$ is not conjugate to a subgroup of a vertex of any indecomposable summand of $V$.
    \end{proof}

    \begin{remark}
        As a result of the previous proposition, if $V$ is a $p$-permutation module and we want to understand ${}_{kG}\calC(V)$, it suffices to understand ${}_{kG}\calC(\calX_V)$ where $\calX_V$ is a set of all vertices of the indecomposable summands of $V$. Equivalently, we may take $\calX_V$ to be the set of all $p$-subgroups $P$ for which $V(P) \neq 0$ (see, for instance, \cite[Proposition 5.10.3]{L181}), since adding subgroups of elements in $\calX_V$ to $\calX_V$ does not change the class of $\calX_V$-projectives. 
    \end{remark}

    \begin{definition}
        Let $V$ be a $kG$-module. Denote by ${}_{kG}\triv(V)$ the full subcategory of ${}_{kG}\catmod(V)$ consisting of $p$-permutation modules, and denote by $Ch^b({}_{kG}\triv,V)$ the subcategory of $Ch^b({}_{kG}\catmod,V)$ consisting of chain complexes whose components are $p$-permutation modules. Let ${}_{kG}p\calC(V)$ represent either ${}_{kG}\triv(V)$ or $Ch^b({}_{kG}\triv,V)$. It is an easy verification that Proposition \ref{indrespreservevproj} and Theorem \ref{cor:ppermrelprojissubgroup} hold with ${}_{kG}p\calC(V)$ replacing ${}_{kG}\calC(V)$.

    \end{definition}

    We do not necessarily assume $V$ is $p$-permutation for ${}_{kG}p\calC(V)$, but the following proposition asserts that in fact, for ${}_{kG}\triv(V)$, we can replace $V$ with a suitable $p$-permutation module without altering the category. Therefore, relative projectivity in ${}_{kG}\triv$ is always equivalent to relative projectivity of with respect to a family of subgroups of $G$ (which can be assumed to be a collection of $p$-subgroups closed under conjugation and taking subgroups).

    \begin{prop}\label{replacewithppermformodules}
        Let $V$ be a $kG$-module. Then, there exists a $p$-permutation module $W$ such that ${}_{kG} \triv(V) = {}_{kG}\triv(W)$. In particular, there exists a poset of $p$-subgroups $\calX\subseteq s_p(G)$ closed under $G$-conjugation and taking subgroups for which ${}_{kG}\triv(V) = {}_{kG}\triv(\calX)$.
    \end{prop}
    \begin{proof}
        First we show that if $M$ is a trivial source module with vertex $P \in s_p(G)$ which belongs to ${}_{kG}\triv(V)$, then $k[G/P]$ belongs to ${}_{kG}\triv(V)$. If $M \mid V\otimes_k N$ for some $kG$-module $N$, then $\Res^G_S M \mid \Res^G_S (V\otimes_k N)$ for $S \in \Syl_p(G)$. Since $M$ has vertex $P$, $\Res^G_S M$ has a direct summand isomorphic to $k[S/P]$, and any other nonisomorphic direct summand is isomorphic to $k[S/Q]$ for some subgroup $Q$ conjugate to a subgroup of $P$. Therefore $\Res^G_S (V\otimes_k N)$ has a direct summand isomorphic to $k[S/P]$. We have $\Res^G_S (V\otimes_k N) \cong  k[S/P] \oplus N'$, for some $kS$-module $N$, therefore, $k[S/P]$ is $\Res^G_S V$-projective. Therefore by \cite[Proposition 2.2.2]{CL12}, ${}_{kS}\triv(P) \subseteq {}_{kS}\triv(\Res^G_S V)$.

        We have that \[\Res^G_S k[G/P] \cong \bigoplus_{x \in [S\backslash G/P]} k[S/S\cap {}^xP],\] and for any $x \in G$, $k[S/S\cap {}^xP]$ is $P$-projective. Therefore, $\Res^G_S k[G/P] \in {}_{kS}\triv(P)$, hence $\Res^G_S k[G/P] \in {}_{kS}\triv(\Res^G_S V)$ as well. Equivalently, $k[G/P]$ belongs to ${}_{kG}\catmod(V)$, as desired.

        Now, let $\calX$ be the set of all vertices which occur for all trivial source modules belonging to ${}_{kG}\triv(V)$. The previous claim and \cite[Proposition 2.2.2]{CL12} imply that \[{}_{kG}\triv(V) \supset {}_{kG}\triv\left(\bigoplus_{P \in \calX} k[G/P]\right) = {}_{kG}\triv(\calX).\]

        The converse inclusion follows from Proposition \ref{prop:3.5}.
    \end{proof}

    In particular, if $V \neq 0$, then ${}_{kG}\triv(V)$ is nonempty. If ${}_{kG}\triv(V) = {}_{kG}\triv(\calX)$ for some subposet $\calX \subseteq s_p(G)$, closed under $G$-conjugation and taking subgroups, then for every $P \in \calX$, the transitive permutation module $k[G/P]$ and all its direct summands are objects of ${}_{kG}\triv(V)$.

    \subsection{Absolute $p$-divisibility}

    The following observation of Auslander and Carlson in \cite{AC86} is crucial for the theory of $V$-projectivity: if $V$ is a $kG$-module, then the \textit{evaluation map} or \textit{trace map} $V^*\otimes_k V \to k: f\otimes v \mapsto f(v)$ is always $V$-split, and splits when $\dim_k V$ is coprime to $p$. The following arguments and definition are taken from \cite{BC86}.

    \begin{theorem}{\cite[2.1]{BC86}}
        Let $k$ be an algebraically closed field of characteristic $p$ (possibly $p = 0$). Let $M,N$ be finite-dimensional indecomposable $kG$-modules, then
        \[k \mid M\otimes_k N \text{ if and only if }  M \cong N^* \text{ and } p \nmid \dim_k(N)  \]
        Moreover, if $k$ is a direct summand of $N^* \otimes_k N$, it has multiplicity one.
    \end{theorem}

    \begin{prop}
        Let $V$ be a $kG$-module. The following are equivalent:
        \begin{enumerate}
            \item The trivial $kG$-module $k$ is relatively $V$-projective.
            \item The chain complex $k[0]$ is relatively $V$-projective.
            \item $p$ does not divide the $k$-dimension of at least one of the indecomposable direct summands of $V$.
            \item ${}_{kG} \catmod(V) = {}_{kG}\catmod.$
            \item $Ch^b({}_{kG}\catmod,V) = Ch^b({}_{kG}\catmod)$.
        \end{enumerate}
    \end{prop}
    \begin{proof}
        The equivalence of (a),(c),(d) is given by \cite[Lemma 2.2.2]{CL11}. Now, $k[0]$ is $V$-projective if and only if  $k \mid V^* \otimes_k V$, and equivalently, $p \nmid \dim_k(V)$. This shows (b) is equivalent to (c). Now, (c) implies (e) by Proposition \cite[Proposition 2.2.2]{CL12}, and (e) implies (b) trivially, so we are done.
    \end{proof}

    In other words, relative $V$-projectivity is interesting if and only if the $k$-dimensions of all indecomposable summands of $V$ are divisible by $p$, as otherwise, all objects are $V$-projective. Analogously, if $S \in \Syl_p(G)$, all $kG$-modules or chain complexes of $kG$-modules are $S$-projective.

    \begin{definition}{\cite{BC86}}
        A $kG$-module for which the $k$-dimensions of all indecomposable summands of $V$ are divisible by $p$ is called \textit{absolutely $p$-divisible.}
    \end{definition}

    \begin{lemma}{\cite[Lemma 2.2.4]{CL11}}
        Let $S \in \Syl_p(G)$, and let $V$ be a $kG$-module with vertex $Q \leq S$.
        \begin{enumerate}
            \item For every subgroup $H \geq S$, $V$ is absolutely $p$-divisible if and only if $\Res^G_H V$ is absolutely $p$-divisible.
            \item If $Q < S,$ then for every subgroup $R \leq S$ with $S \geq R > Q$, $V$ is absolutely $p$-divisible if and only if $\Res^G_R V$ is.
        \end{enumerate}

    \end{lemma}

    In particular, if $V$ is a $p$-permutation module, $V$ is absolutely $p$-divisible if and only if $V$ has no indecomposable direct summands with vertex $S$ if and only if $V(S) = 0$. Moreover, absolute $p$-divisibility is preserved upon restriction to a subgroup containing a Sylow $p$-subgroup. This will be an important fact in the sequel when we consider restriction to subgroups containing Sylow $p$-subgroups.

    %%%%%%%%%%%%%%%%%%%%%%%%%%%%%%%%%%%%%%%%%%%%%%%%%%%%%%%%%%%%%%%%%%%%%%%%%%%%%%%%%%%%%%%%%

    \section{$V$-stable categories}

    Next, we review the stable categories corresponding to $V$-projectivity for modules and chain complexes. These categories are analogous to the usual stable module category, in that we factor out $V$-projective objects.

    \begin{definition}
        Let $V$ be a $kG$-module.
        \begin{enumerate}
            \item The \textit{$V$-stable $kG$-module category}, denoted ${}_{kG}\stmod_V$, is defined as follows. The objects are $kG$-modules, and for any two $kG$-modules $M,N$, the space of morphisms from $M$ to $N$ is the quotient of $k$-vector spaces \[\underline{\Hom}^V_{kG}(M,N) = \Hom_{kG}(M,N)/ \Hom_{kG}^{V}(M,N),\] where $\Hom_{kG}^V(M,N)$ is the subspace of $\Hom_{kG}(U,V)$ consisting of morphisms which factor through a $V$-projective module. Note that this definition makes sense since $\Hom_{kG}^{V}(M,N)$ is a two-sided ideal of $\Hom_{kG}(M,N)$.
            \item The \textit{$V$-stable $p$-permutation $kG$-module category}, denoted ${}_{kG}\sttriv_V$, is defined as the full subcategory of ${}_{kG}\stmod_V$ with $p$-permutation modules as its objects.
        \end{enumerate}

        We have an analogous definition for categories of chain complexes.
        \begin{enumerate}
            \item The \textit{$V$-stable category of chain complexes of $kG$-modules}, denoted $\stCh^b({}_{kG}\catmod)_V$, is defined as follows. The objects of $\stCh^b({}_{kG}\catmod)_V$ are the objects of $Ch^b({}_{kG}\catmod)$, and for any two $C, D \in Ch^b({}_{kG}\catmod)$, \[\underline{\Hom}^V_{kG}(C,D) = \Hom_{kG}(C,D)/ \Hom_{kG}^{V}(C,D),\] where $\Hom_{kG}^V(C,D)$ is the $k$-subspace of $\Hom_{kG}(C,D)$ consisting of morphisms which factor through a $V$-projective chain complex.
            \item The \textit{$V$-stable category of chain complexes of $p$-permutation $kG$-modules}, denoted $\stCh^b({}_{kG}\textbf{triv})_V$, is defined as the full subcategory of $\stCh^b({}_{kG}\catmod)_V$ with objects chain complexes of $p$-permutation modules.
        \end{enumerate}
    \end{definition}

    \textbf{Notation:} Following the previous section, we let ${}_{kG}\calC$ refer to either ${}_{kG}\catmod$ or $Ch^b({}_{kG}\catmod)$. We write ${}_{kG}\stC_{V}$ for the corresponding $V$-stable category ${}_{kG}\stmod_V$ or $\stCh^b({}_{kG}\catmod)_V$. For any morphism $f \in \Hom_{{}_{kG}\calC}(M,N)$, we write $\underline{f}$ to denote its image in ${}_{kG}\stC_{V}$. We denote the morphism sets of ${}_{kG}\stC_{V}$ by $\underline\Hom^V_{{}_{kG}\calC}(M,N)$. Similarly, we let ${}_{kG}p\calC$ refer to either ${}_{kG}\triv$ or $Ch^b({}_{kG}\triv)$. We let ${}_{kG}p\stC_{V}$ refer to the corresponding $V$-stable category ${}_{kG}\sttriv_V$ or $\stCh^b({}_{kG}\triv)_V$, and use the analogous notation for morphisms, with ${}_{kG}p\calC$ replacing ${}_{kG}\calC$. If $V = \{0\}$, we vacuously set ${}_{kG}\stmod_V = {}_{kG}\catmod$ and ${}_{kG}\sttriv_V = {}_{kG}\triv$.

    \begin{prop}
        Let $M\in {}_{kG}\calC$. The following are equivalent.
        \begin{enumerate}
            \item $M$ is isomorphic to the zero object in ${}_{kG}\stC_{V}$.
            \item $\underline{\id}_M$ is the zero endomorphism of $M$ in $\underline{\End}^V_{{}_{kG}\calC}(M)$.
            \item $M \in {}_{kG}\calC(V)$.
        \end{enumerate}
    \end{prop}
    \begin{proof}
        (a) implies (b) is trivial since the zero object only has the zero endomorphism. If (b) holds, then $\id_M = g\circ f$ for some ${}_{kG}\calC$ morphisms $f: M \to P$ and $g: P \to M$, with $P$ a $V$-projective object in ${}_{kG}\calC$. In particular, $g$ is split surjective, so $M$ is isomorphic to a direct summand of $P$, hence $M$ is $V$-projective as well. Finally, (c) implies (a) is straightforward.
    \end{proof}

    The above proposition also holds for ${}_{kG}p\calC$, as it is a full subcategory of ${}_{kG}\calC$. We previously noted that projectivity relative to modules was compatible with induction, restriction, inflation, and the Brauer construction. We translate these compatibilities for the $V$-stable categories constructed here.

    \begin{prop}
        Let $H \leq G$, $P \in s_p(G)$, and $N \trianglelefteq G$.
        \begin{enumerate}
            \item If $V$ is a $kG$-module, $\Res^G_H$ induces functors ${}_{kG}\stC_{V} \to {}_{kH}\stC_{ \Res^G_H V}$ and ${}_{kG}p\stC_{V} \to {}_{kH}p\stC_{\Res^G_H V}$.
            \item If $V$ is a $kH$-module, $\Ind^G_H$ induces functors ${}_{kH}\stC_{V} \to {}_{kG}\stC_{ \Ind^G_H V}$ and ${}_{kH}p\stC_{V} \to {}_{kG}p\stC_{G, \Ind^G_H V}$.
            \item If $V$ is a $k[G/N]$-module, $\Inf^G_{G/N}$ induces functors $_{k[G/N]}\stC_{V} \to {}_{kG}\stC_{\Inf^G_{G/N} V}$ and ${}_{k[G/N]}p\stC_{V} \to {}_{kG}p\stC_{ \Inf^G_{G/N} V}$ .
            \item If $V$ is a $kG$-module, then $-(P)$ induces a functor ${}_{kG}p\stC_{V} \to {}_{k[N_G(P)/P]}p\stC_{V(P)}$, and if $V$ is $p$-permutation, $-(P)$ induces a functor ${}_{kG}\stC_{V} \to {}_{k[N_G(P)/P]}\stC_{V(P)}$.
        \end{enumerate}
    \end{prop}
    \begin{proof}
        This follows from Lemma \ref{indrespreservevproj}.
    \end{proof}

    In particular, if $P$ is a $p$-subgroup of $G$ for which $V(P) = 0$, then $-(P)$ induces a functor ${}_{kG}p\stC_V \to {}_{k[N_G(P)/P]}p\calC$. Recall that in the usual stable module category ${}_{kG}\stmod$, one can detect when a morphism factors through a projective object. We have an analogous result here, which in the case of $V = kG$ reduces to the usual detection theorem.

    \begin{prop}
        Let $V$ be a $kG$-module and let $f: M\to N$ be a ${}_{kG}\calC$-morphism. The following are equivalent:
        \begin{enumerate}
            \item $f$ factors through an object in ${}_{kG}\calC(V)$.
            \item $f$ factors through the morphism $V^*\otimes_k V \otimes_k N \to N$ sending $\delta \otimes v\otimes n$ to $\delta(v)n$ (degreewise, if necessary).
            \item $f$ factors through any $V$-split epimorphism $g: N' \to N$.
        \end{enumerate}
    \end{prop}
    \begin{proof}
        The implications (c) implies (b) implies (a) are straightforward, noting that the homomorphism in (b) is $V$-split. Suppose (a). Then there exists a $V$-projective object $P$ and maps $f_1: M \to P$ and $f_2: P \to N$ such that $f = f_2\circ f_1$. Suppose $g: N' \to N$ is a $V$-split epimorphism as in (c). Therefore, by $V$-projectivity, there exists a morphism $g': P \to N'$ such that $f_2 = g \circ g'$. Thus $f = g\circ g' \circ f_1$, as desired.
    \end{proof}

    \begin{corollary}
        Let $f: M \to N$ be a morphism in ${}_{kG}\triv$ and let $V$ be a $kG$-module for which ${}_{kG}\sttriv_V \neq \emptyset$. Then $f$ factors through an object in ${}_{kG}\catmod(V)$ if and only if $f$ factors through an object in ${}_{kG} \triv(V)$.
    \end{corollary}
    \begin{proof}
        We may replace $V$ by a $p$-permutation module by Proposition \ref{replacewithppermformodules}. By the previous proposition, $f$ factors through an object in ${}_{kG}\catmod(V)$ if and only if $f$ factors through $V^*\otimes_k V\otimes N$, which is $p$-permutation.
    \end{proof}

    As a result, we can equivalently define ${}_{kG}p\stC_V$ as follows: ${}_{kG}p\stC_V$ has the same objects as ${}_{kG}p\calC$, and for any two objects $M,N \in {}_{kG}p\stC_V$, the space of morphisms from $M$ to $N$ is the quotient of $k$-vector spaces \[\underline{\Hom}^V_{{}_{kG}p\calC}(M,N) = \Hom_{{}_{kG}p\calC}(M,N)/ \Hom_{{}_{kG}p\calC}^{V}(M,N),\] where $\Hom_{{}_{kG}p\calC}^V(M,N)$ is the subspace of $\Hom_{{}_{kG}p\calC}(M,N)$ consisting of morphisms that factor through an object in ${}_{kG}p\calC(V)$.

    \begin{prop}\label{stablecatisomorphisms}
        Let $M, N \in {}_{kG}\calC$, and let $\varphi: M \to N$ be a morphism in ${}_{kG}\calC$. If $\varphi$ is a $V$-split epimorphism, then the image of $\varphi$ in $\underline\Hom_{{}_{kG}\calC}^V(M,N)$ is an isomorphism if and only if $\varphi$ is a split epimorphism and $\ker(\varphi)$ is a $V$-projective object.

        In particular, there exists an isomorphism $M \xrightarrow{\cong} N$ in ${}_{kG}\stC_{V}$ if and only if there exist $P,Q \in {}_{kG}\calC(V)$ such that $M \oplus P \cong N \oplus Q$ in ${}_{kG}\calC$.
    \end{prop}
    \begin{proof}
        Suppose $\varphi$ is $V$-split surjective and $\underline\varphi$ is an isomorphism in ${}_{kG}\stC_{V}$. Then there is a ${}_{kG}\calC$ morphism $\psi: N \to M$ such that $\underline\varphi\circ\underline\psi = \underline{\id}_N$, or equivalently (since both categories are $k$-linear), that $\id_N - \varphi \circ\psi $ factors through a $V$-projective $kG$-module. Equivalently, $\id_N - \varphi \circ\psi $ factors through any $V$-split epimorphism, in particular $\varphi$. Write \[\id_N - \varphi\circ\psi = \varphi \circ \alpha\] for some $\alpha: N \to M$. We have \[\id_N = \varphi \circ(\psi+\alpha),\] hence $\varphi$ is a split epimorphism. It follows that $\ker(\psi)$ is zero in ${}_{kG}\stC_V$ for $\underline\varphi$ to be an isomorphism, so $\ker(\varphi)$ is $V$-projective. The converse of the first statement is straightforward.

        Suppose now that $\underline\varphi$ is an isomorphism. Define the object $P = V^* \otimes_k V \otimes_k N$ and denote by $\pi: P \to V$ the canonical morphism sending $\delta\otimes v\otimes n \mapsto \delta(v)n$ (componentwise if necessary). Then $\varphi \oplus \pi: U\oplus P \to V$ is a $V$-split epimorphism since $\pi$ is, and $\underline{\varphi \oplus \pi}$ is an isomorphism, as $P$ is $V$-projective, hence zero. By the first statement, $\varphi$ is a split epimorphism with $\ker(\varphi)$ $V$-projective, so $M \oplus P \cong V \oplus \ker(\varphi)$ as desired. The converse is straightforward.
    \end{proof}

    Note that the above proposition also holds for ${}_{kG}p\calC(V)$.

    \begin{corollary}
        ${}_{kG}\stC_{V}$ and ${}_{kG}p\stC_{V}$ satisfy the Krull-Schmidt property. In particular, any isomorphism class of objects in ${}_{kG}\stC_V$ and ${}_{kG}p\stC_{V}$ has a unique representative which has no $V$-projective direct summands.
    \end{corollary}
    \begin{proof}
        This follows since ${}_{kG}\calC$ satisfies the Krull-Schmidt property, so given any $M \in {}_{kG}\stC_V$ one may write out its unique direct sum decomposition in $\calC_G$ and remove any $V$-projective summands to obtain a unique decomposition in ${}_{kG}\stC_{V}$. This follows for ${}_{kG}p\stC_{V}$ as well since it is a subcategory of ${}_{kG}\stC_{V}$.
    \end{proof}

    \begin{remark}
        In general, given a homomorphism of $kG$-modules, the image of its kernel in the $V$-stable module category will depend on the choice of representative. Therefore, ${}_{kG}\stmod_V$ and ${}_{kG}\sttriv_V$ are not abelian categories, however, they are still preadditive and $k$-linear. Therefore, one may construct the chain complex categories $Ch^b({}_{kG}\stmod_V)$ and $Ch^b({}_{kG}\sttriv_V)$, as well as the homotopy categories $K^b({}_{kG}\stmod_V)$ and $K^b({}_{kG}\sttriv_V)$. To discuss homological properties, we have to choose specific representatives of chain maps, and in general homological properties will not be well-defined for an equivalence class of complexes.

        Since ${}_{kG}\stmod_V$ satisfies the Krull-Schmidt property, any isomorphism class of chain complexes in $Ch^b({}_{kG}\stmod_V)$ has a unique representative with no $V$-projective modules in any degree. This is obtained by taking any representative and removing all indecomposable $V$-projective summands in each degree and corresponding homomorphisms.
    \end{remark}

    \begin{definition}
        The homotopy category $\stK^b({}_{kG}\catmod)_V$ of $\stCh^b({}_{kG}\catmod)_V$ is defined analogously to the non-stable setting. Two maps $f, g \in \underline\Hom_{kG}(C,D)$ are homotopic if there exists a homotopy $h: C\to D[1]$ for which $\underline{f_i- g_i} = \underline{h_{i-1}\circ c_i + d_{i+1}\circ h_i}$ in $\stCh^b({}_{kG}\catmod)_V$, where $c_i:C_i \to C_{i-1}$ and $d_i: D_i \to D_{i-1}$ denote the $i$th differentials in $C,D$ respectively.
    \end{definition}

    It is clear that in $\stK^b({}_{kG}\catmod)_V$, both contractible complexes, i.e. complexes which are isomorphic to the zero object in $K^b({}_{kG}\catmod)$, and $V$-projective complexes are isomorphic to the zero object. These are precisely the complexes that are identified with the zero object.

    \begin{prop} \label{isoclassesofstablehomotopy}
        Let $C$ be a complex of $kG$-modules which is both indecomposable in $\stCh^b({}_{kG}\catmod)_V$ and isomorphic to the zero complex in $\stK^b({}_{kG}\catmod)_V$. Then $C \cong (0 \to M \xrightarrow{\cong} M \to 0)$ in $\stCh^b({}_{kG}\catmod)_V$ for some $kG$-module $M$.

        In particular, the complexes that are isomorphic to the zero complex in $\stK^b({}_{kG}\catmod)_V$ are isomorphic to direct sums of contractible complexes of $kG$-modules and $V$-projective complexes. Two complexes $C,D$ for which $C\cong D$ in $\stK^b({}_{kG}\catmod)_V$ satisfy $C \oplus M \oplus N \cong D \oplus M' \oplus N'$, where $M,M'$ are contractible chain complexes and $N, N'$ are $V$-projective chain complexes.
    \end{prop}
    \begin{proof}
        Let $C$ be as above and let $i \in \Z$ be the maximal integer for which $C_i \neq 0$. Since $C \cong 0$ in $\stK^b({}_{kG}\catmod)_V$, we have a homotopy $\underline{\id} = \underline{h}\circ\underline{d} + \underline{d}\circ\underline{h}$, hence $\id = h\circ d + d\circ h + f$, for some endomorphism $f$ which factors through a $V$-projective complex. In particular, we have $\id_{C_i} = h_{i-1}\circ d_i + f_i$, with $f_i$ a morphism which factors through a $V$-projective module, hence also a $V$-projective complex by setting the module in degree $i$.

        We claim $0\to C_i \xrightarrow{\id} C_i \to 0$ is isomorphic to a direct summand of $C \in \stCh^b({}_{kG}\catmod)_V.$ Indeed, from the homotopy, the following set-up is a well-defined sequence of morphisms which compose to the identity in $\stCh^b({}_{kG}\catmod)_V$.

        \begin{figure}[H]
            \centering
            \begin{tikzcd}
                \dots\ar[r] & 0 \ar[r] \ar[d] & C_i \ar[d, "="] \ar[r, "="] & C_i\ar[r] \ar[d, "d_i"] & 0 \ar[d] \ar[r] & \dots \\
                \dots \ar[r] & 0 \ar[r]
                \ar[d] & C_i \ar[r, "d_i"] \ar[d, "="] & C_{i-1} \ar[r, "d_{i-1}"]\ar[d, "h_{i-1}"] & C_{i-2} \ar[d] \ar[r] & \dots \\
                \dots \ar[r] & 0 \ar[r] & C_i \ar[r, "="] & C_i \ar[r]& 0 \ar[r] & \dots
            \end{tikzcd}
        \end{figure}

        However, $C$ was assumed to be indecomposable, hence $C = 0 \to C_i \to C_i \to 0$ as desired.

        To prove the final statement, suppose that $C \cong D$ in $\stK^b({}_{kG}\catmod)_V$. Then we have a homotopy equivalence $f: C \to D$. Let $P = cone(\id_D)[-1]$, the mapping cone of $\id_D$, and let $p$ be the component-wise projection of $P$ onto $D$. This is a degreewise split epimorphism (but not split), so the chain map $f\oplus p: C\oplus P \to D$ is also degree-wise split. Since $P$ is contractible, we can write $C \oplus P \cong D \oplus \ker(f \oplus d)$ in $\stCh^b({}_{kG}\catmod)_V$. It follows that $\ker(f \oplus d)$ is contractible as well, and the result follows from Theorem \ref{stablecatisomorphisms}.
    \end{proof}

    It follows that we may always take a unique, up to isomorphism, representative $C$ of a class of complexes $[C] \in \stK^b({}_{kG}\catmod)_V$ which has no $V$-projective or contractible summands.

    \section{$V$-endotrivial modules}

    We briefly recall a few necessary facts about relatively $V$-endotrivial modules, which was first introduced by Lassueur in \cite{CL11}.

    \begin{definition}
        Let $V$ be a $kG$-module. A $kG$-module $M$ is \textit{endotrivial relative to $V$} or \textit{$V$-endotrivial} if its $k$-endomorphism ring is the direct sum of a trivial module and a $V$-projective module. That is, $M$ is $V$-endotrivial if and only if \[\End_k(M) \cong M^* \otimes_k M \cong k \oplus P, \quad P \in {}_{kG}\catmod(V).\] Equivalently, $\End_k(M) \cong k \in {}_{kG}\stmod_V$.
    \end{definition}

    Similar to the way that endotrivial modules arise from Heller translates, $V$-endotrivial modules arise from relative syzygies induced by $V$-projective modules. The following notion is due to Carlson, see \cite[Section 8]{C96} for instance.

    \begin{definition}{\cite[Section 8]{C96}}
        A \textit{$V$-projective resolution} of a $kG$-module $M$ is a nonnegative complex $P$ of $V$-projective modules together with a surjective $kG$-module homomorphism $P_0 \xrightarrow{\epsilon} M$ such that the sequence \[\cdots \to P_2 \xrightarrow{d_2} P_1 \xrightarrow{d_1} P_0 \xrightarrow{\epsilon} M \to 0\] is exact and \textit{totally $V$-split}, that is, for all $i \geq 1,$ the short exact sequences \[0 \to \ker (d_i) \to P_i \to \operatorname{im} (d_i) \to 0\] and \[0\to \ker(\epsilon) \to P_0 \xrightarrow{\epsilon} M \to 0 \] are $V$-split.
    \end{definition}

    \begin{remark}
        We have a dual notion of $V$-injective resolutions as well.
        $V$-protective resolutions always exist, since the canonical map $V^* \otimes_k V \otimes_k M \to M$ is $V$-split surjective, and one may iterate this construction. There is also a relative comparison theorem: if $P$ and $Q$ are two $V$-projective resolutions of a $kG$-module $M$, then there is a chain map $\mu: (P_\ast \to M) \to (Q_\ast \to M)$ which lifts to the identity on $M$. Using this, one may demonstrate the existence of a minimal $V$-projective resolution.
    \end{remark}

    \begin{definition}{\cite[Section 8]{C96}}
        Let $M$ be a $kG$-module, $P$ a minimal $V$-projective resolution of $M$ with differentials $d_i$, and $I$ a minimal $V$-injective resolution of $M$ with differentials $d^i$. Define for all $n \geq 1: \Omega_V^n(M) := \ker (d_n)$ and for all $\Omega_V^{-n}(M) := \coker d^{n-1}$. Define $\Omega_V^0(M)$ to be the $V$-projective-free part of $M$.
    \end{definition}

    \begin{prop}{\cite[Lemmas 3.2.1, 3.2.2, 3.3.1, 3.4.1]{CL11}}
        Let $V$ be an absolutely $p$-divisible $kG$-module.
        \begin{enumerate}
            \item Let $P \in {}_{kG}\catmod(V)$ and $0 \to L \to P \to N \to 0$ be a $V$-split exact sequence. Then $N$ is $V$-endotrivial if and only if $L$ is. In particular, if $M$ is $V$-endotrivial and $W\in {}_{kG}\catmod(V)$, then for every $n \in \Z$, $\Omega^n_W(M)$ is $V$-endotrivial.
            \item If $M,N$ are $V$-endotrivial $kG$-modules, so are $M^*, M\otimes_k N$, and $\Hom_k(M,N)$.
            \item Each $V$-endotrivial $kG$-module has a unique indecomposable $V$-endotrivial direct summand, and all other summands are $V$-projective.
            \item If $M$ is an indecomposable $V$-endotrivial $kG$-module, then its vertex set is $\Syl_p(G)$, and given a vertex $S$, any corresponding $kS$-source is $\Res^G_S V$-endotrivial.
        \end{enumerate}

    \end{prop}

    \begin{definition}
        Let $V$ be an absolutely $p$-divisible $kG$-module. The group of $V$-endotrivial complexes $T_{k,V}(G)$ is defined as follows. For two $V$-endotrivial $kG$-modules $M,N$, say $M \sim N$ if and only if $M_0 \cong N_0$, where $M_0$ and $N_0$ are their respective indecomposable $V$-endotrivial direct summands. Equivalently, $M\sim N$ if $M \cong N$ in ${}_{kG}\stmod_V$. Let $T_{k,V}(G)$ denote the resulting set of equivalence classes. In particular, any equivalence class in $T_{k,V}(G)$ consists of an indecomposable $V$-endotrivial module $M_0$, and all modules of the form $M_0 \oplus P$ with $P$ a $V$-projective $kG$-module.

        $T_{k,V}(G)$ forms an abelian group under the operation \[[M] + [N] = [M\otimes_k N],\] with zero element $[k]$ and inverse $-[M] = [M^*].$
    \end{definition}

    \begin{remark}
        If $U, V$ are absolutely $p$-divisible $kG$-modules with ${}_{kG}\catmod(U) \subseteq {}_{kG}\catmod(V)$, then we have an injective group homomorphism $T_{k, U}(G) \to T_{k,V}(G)$ given by $[M]_U \mapsto [M]_V$. Moreover, inflation, restriction, group isomorphisms, and conjugation all induce corresponding group homomorphisms:
        \begin{itemize}
            \item $\Res^G_H: T_{k,V}(G) \to T_{k, \Res^G_H V}(H)$
            \item $\Inf^G_{G/N}: T_{k,V}(G/N) \to T_{k, \Inf^G_{G/N} V}(G)$
            \item $\Iso_f: T_{k,V}(G) \to T_{k, \Iso_f(V)}(G')$
            \item $c_g: T_{k,V}(H) \to T_{k, {}^gV}({}^gH)$
        \end{itemize}
    \end{remark}

    \begin{theorem}\label{trivsourceiskernelofrestriction}
        Let $V$ be an absolutely $p$-divisible $kG$-module, and let $M$ be an indecomposable $V$-endotrivial $kG$-module. Then, $\Res^G_S M \cong k \oplus M'$ for some $\Res^G_S V$-projective $kS$-module $M'$ if and only if $M$ has trivial source. In particular, $T_{k,V}(G,S) := \ker (\Res^G_S)$ consists of equivalence classes of trivial source $V$-endotrivial $kG$-modules.
    \end{theorem}
    \begin{proof}
        The forward direction is given by \cite[Lemma 3.4.1]{CL11}. We prove the reverse direction: suppose that $M$ is a trivial source $V$-endotrivial $kG$-module. Then since $\Res^G_S M$ is $\Res^G_S V$-endotrivial, it contains a unique indecomposable $\Res^G_S V$-endotrivial direct summand. We have a direct sum decomposition $\Res^G_S M = M' \oplus N$, with $M' \in {}_{kG}\catmod(\Res^G_S V)$, and $N$ an indecomposable endotrivial trivial source $kS$-module. Since $N$ has vertex $S$ and is a trivial source module, $N \cong k$, as desired.
    \end{proof}

    In particular, $T_{k,V}(G,S)$ is finite, as there are finitely many trivial source modules for any finite group $G$. In the case of $V = kG$, $T_k(G,S):= T_{k,kG}(G,S)$ has been described in detail by Balmer in \cite{Ba13} and by Grodal in \cite{Gr22} using separate methods. We refer the reader to \cite{CL11} and \cite{CL12} for more details on $V$-endotrivial modules.

    \section{$V$-endotrivial complexes}

    \textbf{Notation: For the rest of this paper,} unless otherwise specified, $V$ is assumed to be an absolutely $p$-divisible $kG$-module.

    \begin{definition}\label{endotrivialdef}
        Let $C \in Ch^b({}_{kG}\triv)$.
        \begin{enumerate}
            \item $C$ is \textit{weakly $V$-endotrivial} if $\End_k(C) \cong C^* \otimes_k C \simeq k[0] \oplus D$, where $D$ is a bounded chain complex of $V$-projective $kG$-modules.
            \item $C$ is \textit{strongly $V$-endotrivial} if $\End_k(C) \cong C^* \otimes_k C \simeq k[0] \oplus D$, where $D$ is a bounded $V$-projective $kG$-chain complex.
            \item $C$ is \textit{$V$-endosplit-trivial} if $\End_k(C) \cong C^* \otimes_k C \simeq k[0] \oplus M[0]$, where $M$ is a $V$-projective $kG$-module. In this case, $M$ is necessarily $p$-permutation as well.
        \end{enumerate}
        We say a chain complex is \textit{$V$-endotrivial} when it satisfies any one of the three definitions. If $V = 0$, we say $C$ is \textit{endotrivial}, in which case all three definitions vacuously coincide. This coincides with the definition of an endotrivial complex given in \cite{SKM23}.
    \end{definition}

    \begin{remark}
        It is clear to see that a $V$-endosplit-trivial complex is strongly $V$-endotrivial, and a strongly $V$-endotrivial complex is weakly $V$-endotrivial. Moreover, $V$-endosplit-trivial complexes are always examples of endosplit $p$-permutation resolutions up to a shift in degree; we will discuss endosplit $p$-permutation resolutions in Section \ref{endosplitppermresolutions}.

        By Proposition \ref{replacewithppermformodules}, given any $V$, there exists a $p$-permutation $kG$-module $W$ such that the classes of weakly $V$-endotrivial (resp. $V$-endosplit-trivial) complexes and the classes of weakly $W$-endotrivial (resp. $W$-endosplit-trivial) complexes coincide. Therefore, when considering weakly $V$-endotrivial complexes or $V$-endosplit-trivial complexes, it suffices to assume $V$ is $p$-permutation, and, more generally, the only types of weak $V$-endotriviality or $V$-endosplit-triviality are those which arise from projectivity relative to subgroups.
    \end{remark}

    For each version of $V$-endotriviality, we have a corresponding group parameterizing it. These definitions will correspond to the category in which they live. Weakly $V$-endotrivial complexes are the invertible objects of $K^b({}_{kG}\sttriv_V)$, while strongly $V$-endotrivial complexes are the invertible objects of $\stK^b({}_{kG}\triv)_V$, and $V$-endosplit-trivial complexes are particularly nice examples of strongly $V$-endotrivial complexes. We will see in the sequel that some of these groups can be defined analogously to how groups of (relatively) endotrivial modules are defined.

    \begin{definition}\label{def:reletrivdef}

        \begin{enumerate}
            \item We define the group of weakly $V$-endotrivial chain complexes $w\calE^V_k(G)$ to be the set of all isomorphism classes of weakly $V$-endotrivial complexes in the category $K^b({}_{kG}\sttriv_V)$. In other words, $w\calE^V_k(G)$ is the Picard group of $K^b({}_{kG}\sttriv_V)$.

            \item We define the group of strongly $V$-endotrivial chain complexes $s\calE^V_k(G)$ to be the set of all isomorphism classes of strongly $V$-endotrivial complexes in the category $\stK^b({}_{kG}\triv)_V$. In other words, $s\calE^V_k(G)$ is the Picard group of $\stK^b({}_{kG}\triv)_V$.

            \item We define the group of $V$-endosplit-trivial chain complexes $e\calE^V_k(G)$ to be the set of all isomorphism classes of $V$-endosplit-trivial complexes in the category $\stK^b({}_{kG}\triv)_V$. By definition, $e\calE^V_k(G)$ is the subgroup of $s\calE^V_k(G)$ consisting of $V$-endosplit-trivial complexes.
        \end{enumerate}

        We write $x\calE^V_k(G)$ to denote any of the above three groups. $x\calE^V_k(G)$ forms an abelian group with group addition induced from $\otimes_k$ and inverse induced by taking duals. For a $V$-endotrivial complex $C$, let $[C]\in x\calE^V_k(G)$ denote the corresponding class in the group. We set $\calE_k(G) := e\calE^{0}_k(G)$. One may easily verify that $\calE_k(G)$ coincides with the construction of $\calE_k(G)$ in \cite{SKM23}.
    \end{definition}

    If $U,V$ are two $kG$ modules that satisfy ${}_{kG}\catmod(U) \subseteq {}_{kG}\catmod(V)$, then we have a group homomorphism $x\calE^{U}_k(G) \to x\calE^{V}_k(G)$. This is injective when $x = s$ or $e$.

    \begin{prop}\label{inclusionintolargerendotrivialgroupisinjective}
        Let $U,V$ be absolutely $p$-divisible $kG$-modules (allowing $U = 0$) such that ${}_{kG}\catmod(U) \subseteq {}_{kG}\catmod(V)$. The homomorphism $\iota: x\calE_k^U(G) \to x\calE_k^V(G)$ induced by $[C] \mapsto [C]$ is well-defined, and if $x = s$ or $e$, then $\iota$ is injective.
    \end{prop}
    \begin{proof}
        The map $\iota$ is well-defined, since if two relatively endotrivial chain complexes $C_1,C_2$ are isomorphic in $K^b({}_{kG}\underline\catmod_U)$ (resp. $\stK^b({}_{kG}\catmod)_U$), they are isomorphic in $k^b({}_{kG}\stmod_V)$ (resp. $\stK^b({}_{kG}\catmod)_V$).

        Assume $x = s$ or $e$. $\ker(\iota)$ consists of all equivalence classes of strongly $U$-endotrivial (resp. $U$-endosplit-trivial) complexes which, upon being considered strongly $V$-endotrivial (resp. $V$-endosplit-trivial), have indecomposable endotrivial summand isomorphic to $k[0]$. However, it is straightforward to see that this is simply the equivalence class $[k[0]]$, so $\ker(\iota)$ is trivial.
    \end{proof}

    \begin{remark}
        In particular, we have group homomorphisms $\iota: \calE_k(G) \to x\calE_k^V(G)$ induced by considering an endotrivial $kG$-complex as a $V$-endotrivial $kG$-complex. If $x = e$ or $s$, this is an inclusion. It is not, however, for $x = w$. For example, if $G = C_2$ and $V = kC_2$, the endotrivial chain complex $kC_2 \twoheadrightarrow k$ with $k$ in degree 0 is identified with the trivial weakly $V$-endotrivial chain complex $k[0]$.

        Additionally, we have a group homomorphism $s\calE^V_k(G) \to w\calE^V_k(G)$ induced by the functor $\stK^b({}_{kG}\triv)_V \to K^b({}_{kG}\sttriv_V)$. This homomorphism is in general not injective, as the example given in the previous paragraph demonstrates. Finally, we have an injective homomorphism $T_{k,V}(G,S) \to x\calE_k^V(G)$ given by $[M] \mapsto [M[0]]$.
    \end{remark}

    \begin{prop}\label{endotrivialsummand}
        If $C$ is a weakly $V$-endotrivial (resp. strongly $V$-endotrivial, $V$-endosplit-trivial) complex of $kG$-modules. $C$ has a unique indecomposable direct summand (in $Ch^b({}_{kG}\triv)$) which is weakly $V$-endotrivial (resp. strongly $V$-endotrivial, $V$-endosplit-trivial), and all other direct summands of $C$ are isomorphic to the zero complex in $K^b({}_{kG}\sttriv_V)$ (resp. $\stK^b({}_{kG}\triv)_V$, resp. $\stK^b({}_{kG}\triv)_V$).
    \end{prop}
    \begin{proof}
        Let ${}_{kG}\stC_V$ denote either $K^b({}_{kG}\sttriv_V)$ or $\stK^b({}_{kG}\triv)_V$. Suppose $C$ is $V$-endotrivial and is decomposable. Write $C = C_1 \oplus C_2$, then \[(C_1 \oplus C_2)^* \otimes_k (C_1 \oplus C_2) \cong \End_k(C_1) \oplus \End_k(C_2) \oplus (C_1^* \otimes_k C_2) \oplus (C_1^* \otimes_k C_2)^* \cong  k[0]\] in ${}_{kG}\stC_V$. $(C_1 \otimes_k C_2^*)$ cannot have a direct summand isomorphic to $k$ in ${}_{kG}\stC_V$, since if it did, so would its dual. Therefore, exactly one value of $i \in \{1, 2\}$ satisfies $k \mid \End_k(C_i)$, and all other terms in the above isomorphism are isomorphic to the zero complex (denoted by 0) in ${}_{kG}\stC_V$. Without loss of generality, assume $i = 1$. Then, \[C_2 \cong (C_1 \otimes_k C_1^*)\otimes_k C_2 \cong C_1 \otimes_k (C_1^* \otimes_k C_2) \cong C_1 \otimes_k 0 \cong 0.\] Thus $C_2$ is the zero complex in ${}_{kG}\stC_V$. It follows that $\End_k(C_1) \cong k$ in ${}_{kG}\stC_V$, and repeating this argument until $C_1$ is indecomposable yields the result.
    \end{proof}

    \begin{definition}
        Let $C$ be a $V$-endotrivial complex. We call the unique indecomposable $V$-endotrivial direct summand of $C$ the \textit{cap} of $C$.
    \end{definition}

    \begin{remark}\label{equivdefinitionofgroups}
        \begin{enumerate}
            \item It is important to note that an element $[C] \in w\calE_k^V(G)$ may have more than one representative which is indecomposable as a chain complex of $kG$-modules. On the other hand, the isomorphism classes of chain complexes in a single isomorphism class of $K^b({}_{kG}\sttriv_V)$ is more easily understood due to Theorem \ref{isoclassesofstablehomotopy}. In particular, given any $[C] \in s\calE_k^V(G)$, there is a unique indecomposable chain complex representative $C$ of $[C].$

            \item From the previous proposition, we obtain an equivalent definition of the groups $s\calE_k^V(G)$ and $e\calE_k^V(G)$ which is independent of $\stK^b({}_{kG}\triv)_V$ and analogous to the construction of the Dade group. We impose the following equivalence relation on the set of strongly $V$-endotrivial (resp. $V$-endosplit-trivial) complexes as follows: given two strongly $V$-endotrivial (resp. $V$-endosplit-trivial) complexes $C,D$, we write $C \sim D$ if their caps $C_0, D_0$ satisfy $C_0 \cong D_0$ as chain complexes. Define the group $s\calE_k^V(G)$ (resp. $e\calE_k^V(G)$) to be the group generated by these equivalences classes, with group operation induced by $\otimes_k$. It is a straightforward verification to check that these two group definitions coincide.
        \end{enumerate}
    \end{remark}

    Since restriction, inflation, and the Brauer construction are additive functors and naturally commute with tensor products of $p$-permutation modules, the following proposition is a direct consequence of Proposition \ref{indrespreservevproj}.

    \begin{prop}
        Let $H \leq G, N \trianglelefteq G,$ and $P \in s_p(G)$.
        \begin{enumerate}
            \item If $C$ is a $V$-endotrivial complex of $kG$-modules, $\Res^G_H C$ is a $\Res^G_H V$-endotrivial complex of $kH$-modules, and restriction induces a group homomorphism $x\calE_k^V(G) \to x\calE_k^{\Res^G_H V}(H)$.
            \item Let $V$ be an absolutely $p$-divisible $k[G/N]$-module. If $C$ is a $V$-endotrivial complex of $k[G/N]$-modules, $\Inf^G_{G/N} C$ is a $\Inf^G_{G/N} V$-endotrivial complex of $kG$-modules, and and inflation induces an injective group homomorphism $x\calE_k^V(G/N) \to x\calE_k^{\Inf^G_{G/N} V}(G)$.
            \item If $C$ is a $V$-endotrivial complex of $kG$-modules, $C(P)$ is a $V(P)$-endotrivial complex of $k[N_G(P)/P]$-modules for any $P \in s_P(G)$, and the Brauer construction induces a group homomorphism $x\calE_k^V(G) \to x\calE_k^{V(P)}(N_G(P)/P)$. In particular, if $V(P) = 0$, the Brauer construction induces a group homomorphism $x\calE_k^V(G) \to \calE_k(N_G(P)/P)$.
        \end{enumerate}
    \end{prop}

    It is clear that inflation induces an injective group homomorphism, but the kernel and cokernel of the homomorphisms induced by restriction and the Brauer construction are more challenging to compute. We will study restriction to subgroups containing a Sylow $p$-subgroup in the sequel. Note that $\Res^G_H V$ or $V(P)$ may not be absolutely $p$-divisible. On the other hand, inflation always preserves absolute $p$-divisibility.

    \begin{prop}\label{vtxset}
        Let $V$ be an absolutely $p$-divisible $kG$-module, and let $C$ be an indecomposable $V$-endotrivial complex. Then $C$ has vertex set $\Syl_p(G)$.
    \end{prop}
    \begin{proof}
        It suffices to show this for weakly $V$-endotrivial complexes. $C(S) \neq 0$ for all $S \in \Syl_p(G)$ since $V$ is not $p$-divisible, therefore by properties of $p$-permutation modules, has components which are not $P$-projective for any $P < S$. Observe that if a chain complex $C$ is $H$-projective, then each of its components is $H$-projective as well. $C$ contains at least one indecomposable $p$-permutation module with vertex $S$ since $C(S)$ is nonzero (as $C(S) \otimes_k C(S)^* \simeq k$), so we are done.
    \end{proof}

    We will leave aside arbitrary strongly $V$-endotrivial complexes in this paper, choosing to focus on the subclass of $V$-endosplit-trivial complexes, which appear to be the most well-behaved examples of invertible objects of $\stK^b({}_{kG}\triv)_V$. Attempting to classify all of the invertible objects of $\stK^b({}_{kG}\triv)_V$ may be a possible subject of study for future research.

    \section{An equivalent definition of weak $V$-endotriviality}

    Recall that by Proposition \ref{replacewithppermformodules}, for weak $V$-endotriviality and $V$-endosplit-triviality, we may always assume $V$ is $p$-permutation. We will prove an important equivalent formulation of weak $V$-endotriviality, assuming that $V$ is $p$-permutation.

    \begin{theorem}{\cite[Theorem 5.6]{SKM23}}\label{boucthmgeneralization}
        Let $C$ be a chain complex of $p$-permutation $kG$-modules, and let $\mathcal{X}$ be a subset of the $p$-subgroups of $G$ which is closed under $G$-conjugation and taking subgroups. The following are equivalent:
        \begin{enumerate}
            \item For all $P \not\in\mathcal{X}$, $C(P)$ is acyclic.
            \item There exists a chain complex $D$ with $C\simeq D$ such that for all $i \in \Z$, $D_i$ is $\calX$-projective or zero.
        \end{enumerate}
    \end{theorem}

    \begin{theorem}\label{equivdefinitionweakendotrivial}
        Let $C \in Ch^b({}_{kG}\triv)$, let $V$ be a $p$-permutation $kG$-module which is absolutely $p$-divisible, and let $\calX_V$ be the set of $p$-subgroups $P$ of $G$ for which $V(P) = 0$. Then $C$ is weakly $V$-endotrivial if and only if $C(P)$ has nonzero homology concentrated in exactly one degree for all $P \in \calX_V$, with the nontrivial homology having $k$-dimension one.
    \end{theorem}
    \begin{proof}
        Note that since $V$ is absolutely $p$-divisible, $\calX_V$ contains Sylow $p$-subgroups. First, suppose $C$ is $V$-endotrivial, i.e. $C \otimes_k C^* \simeq k[0] \oplus D$, for some chain complex $D$ which contains only $V$-projective modules. Note that for any $V$-projective module $M$ and $P \in \calX_V$, $M(P) = 0$. For any $p$-subgroup $P \in \calX_V$, we have \[k[0] \cong (k[0]\oplus D)(P) \simeq (C\otimes_k C^*)(P)\cong C(P)\otimes_k (C^*)(P)\cong C(P) \otimes_k (C(P))^*.\] It follows by the K\"unneth formula that $C(P)$ has nonzero homology concentrated in one degree with $k$-dimension one.

        Conversely, suppose that for all $p$-subgroups $P \in \calX_V$, $C(P)$ has homology concentrated in one degree with $k$-dimension one. We have the tensor-hom adjunction \[\Hom_{kG}(C \otimes_k -, -) \cong \Hom_{kG}(-, \Hom_k(C, -)) ,\] and taking the unit of the adjunction we obtain an injective chain complex homomorphism $\phi: k[0] \to \Hom_{kG}(C, k[0])\otimes_k C \cong C^* \otimes_k C$. Let $D$ be the mapping cone of $\phi$. For $P \in \calX_V$, it follows by assumption that $\phi(P)$ induces a homotopy equivalence $k[0] \simeq (C^*\otimes_k C)(P)$, so in particular, $D(P)$ is contractible. Therefore, by Theorem \ref{boucthmgeneralization}, $D$ is homotopy equivalent to a chain complex of $(s_p(G) \setminus \calX_V)$-projective modules, or equivalently, $V$-projective modules. Since $S \in\calX_V$, it follows by considering the mapping cone structure that $D$ has as a direct summand the chain complex $k \xrightarrow{\sim} k$, implying $C\otimes_k C^* \simeq k \oplus E$, for some bounded chain complex $E$ containing only $V$-projective modules.
    \end{proof}

    \begin{remark}
        \begin{enumerate}
            \item In fact, one only needs to check that $C(P)$ has nonzero homology concentrated in one degree, with that homology having $k$-dimension one at a \textit{minimal} subgroup $P \in \calX_V$, with order relation the inclusion of subgroups, to verify that $C$ is weakly $V$-endotrivial. This follows from the property (see \cite[Proposition 5.8.5]{L181}) that if $P \trianglelefteq Q$ are $p$-subgroups, then $C(P)(Q) \cong C(Q)$ as chain complexes of $k[N_{N_G(P)}(Q)]$-modules. Therefore, if $P, Q$ are $p$-subgroups of $G$ with $[Q:P] = p$ and $C(P)$ has nonzero homology concentrated in one degree with $k$-dimension one, then $C(Q)$ has this property as well.
            \item The prior theorem is analogous to a result of Rouquier, \cite[Theorem 5.6]{RR01}, which provides a local-to-global criteria for determining when a chain complex of $p$-permutation bimodules induces a splendid stable equivalence of block algebras.
        \end{enumerate}
    \end{remark}

    In fact, weakly $kG$-endotrivial complexes induce splendid stable equivalences. Denote by $\Delta G \leq G\times G$ the subgroup $\{(g,g) \in G\times G \mid g\in G\}$.

    \begin{theorem}\label{splendidstableequiv}
        Let $C$ be a weakly $kG$-endotrivial complex of $kG$-modules, which is identified as a complex of $k\Delta G$-modules. Then $\Ind^{G\times G}_{\Delta G} C$, regarded as a complex of $(kG, kG)$-bimodules, induces a splendid stable equivalence on $kG$.
    \end{theorem}
    \begin{proof}
        The proof follows in the same manner as the series of proofs given in \cite[Section 4]{SKM23} that endotrivial complexes induce splendid Rickard autoequivalences.
    \end{proof}

    \section{Endosplit $p$-permutation resolutions} \label{endosplitppermresolutions}

    In \cite[Section 7]{R96}, Rickard introduced the notion of an \textit{endosplit $p$-permutation resolution} (in which they are termed endosplit permutation resolutions). $V$-endosplit-trivial complexes are, after a shift of degree, examples of endosplit $p$-permutation resolutions, hence the naming convention. We will review some features of endosplit $p$-permutation resolutions, and give an equivalent condition for a chain complex to be an endosplit $p$-permutation resolution.

    \begin{definition}{\cite[Definition 7.1]{R96}}
        let $M$ be a $kG$-module. A chain complex of $p$-permutation modules $C$ is said to be an \textit{endosplit $p$-permutation resolution of $M$} if $H_0(C) \cong M$, $H_i(C) = 0$ for all $i \neq 0$, and $C^* \otimes_k C$ is split as a chain complex of $kG$-modules.
    \end{definition}

    \begin{remark}
        \begin{enumerate}
            \item In particular, the K\"unneth formula implies $C^* \otimes_k C \simeq (M^*\otimes_k M)[0],$ since a split chain complex is homotopy equivalent to the chain complex with its homology in each degree and zero differentials.

            Suppose $M$ is a $V$-endotrivial module, as defined in \cite{CL11}, and suppose $C$ is an endosplit $p$-permutation resolution for $M$. We have $C^* \otimes_k C \simeq (M^*\otimes_k M)[0]\cong (k \oplus P)[0]$ for some $V$-projective $kG$-module $P$, hence $C$ is a $V$-endosplit-trivial complex. Conversely, it is easy to see that any $V$-endosplit-trivial complex is, up to a shift in degrees, an endosplit $p$-permutation resolution, and it follows that $V$-endosplit-trivial complexes are equivalently shifted endosplit $p$-permutation resolutions for $V$-endotrivial modules. Moreover, if $M$ is one-dimensional, then $C$ is an endotrivial complex.

            \item Rickard proved in \cite{R96} that endosplit $p$-permutation resolutions exist for any endopermutation $kP$-module, where $P$ is an abelian $p$-group. Using the classification of endotrivial modules for $p$-groups, Mazza showed in \cite{Ma03} that if $p$ is odd, every endotrivial $kP$-module has an endosplit $p$-permutation resolution. On the other hand, some endotrivial modules such as the exotic endotrivial $kQ_8$-module does not have an endosplit $p$-permutation resolution.
        \end{enumerate}

    \end{remark}

    We have an equivalent characterization of endosplit $p$-permutation resolutions as follows.

    \begin{prop}\label{prop:equivendosplit}
        Let $C$ be a bounded chain complex of $p$-permutation $kG$-modules. The following are equivalent:
        \begin{enumerate}
            \item $C$ is a shifted endosplit $p$-permutation resolution.
            \item For every $P \in s_p(G)$, one of the following holds: either $C(P)$ is contractible or $C(P)$ has nonzero homology in exactly one degree. 
        \end{enumerate}
    \end{prop}
    \begin{proof}
        First, suppose $C$ is a shifted endosplit $p$-permutation resolution for some $kG$-module $M$. Fix $P \in s_p(G)$. $\End_k(C) \cong C^* \otimes_k C$ is split, and the K\"unneth formula implies $C^* \otimes_k C$ has nonzero homology concentrated in degree zero, with $H_0(C^* \otimes_k C) \cong M^* \otimes_k M$, therefore \[C^* \otimes_k C\simeq (M^* \otimes_k M)[0].\] Applying the Brauer quotient, we obtain \[C(P)^* \otimes_k C(P) \cong (C^* \otimes_k C)(P) \simeq (M^* \otimes_k M)(P)[0].\] If $(M^* \otimes_k M)(P) \neq 0$, it follows again by the K\"unneth formula that $C(P)$ has nonzero homology in exactly one degree. Otherwise, if $(M^* \otimes_k M)(P) = 0$, then $C(P)^*\otimes_k C(P)\cong \End_k(C(P))$ is contractible. Thus, $\End_{K^b({}_{k[N_G(P)/P]}\triv)}(C(P)) = 0$, and $C(P)$ is isomorphic to the zero complex in $K^b({}_{kG}\triv)$, as desired.

        Now, suppose condition (b) holds. Suppose for contradiction that $C^* \otimes_k C \not\simeq M[0]$, then we may write $C^* \otimes_k C \simeq D$ where $D$ is the unique direct summand of $D$ with no contractible summands. Since $C$ is bounded and $C^* \cong C$, there exists a unique positive natural number $i > 0$ for which $D_i \neq 0$ and $D_j = 0$ for all $j > i$. Since $C(P)$ is either contractible or has nonzero homology in exactly one degree, $C(P)^* \otimes_k C(P)$ has nonzero homology concentrated in degree 0, possibly also zero. The same holds for $D(P)$, as it is a direct summand of $(C^* \otimes_k C)(P) \cong  C(P)^* \otimes_k C(P).$ In particular, $d_i(P): D_i(P) \to D_{i-1}(P)$ is injective for all $P \in s_p(G)$, so $d_i: D_i \to D_{i-1}$ is split injective. Therefore, $D$ has a contractible direct summand, a contradiction. Thus, $C^* \otimes_k C \not\simeq M[0]$ for some $kG$-module $M$. However, the K\"unneth formula implies there exists a unique $i\in \Z$ for which $H_i(C)\neq 0$. We conclude $C$ is a shifted endosplit $p$-permutation resolution.
    \end{proof}

    Since $V$-endosplit-trivial complexes are endosplit $p$-permutation resolutions, we obtain the following equivalent characterizations.

    \begin{corollary}\label{cor:vesplittrivaltdefs}
        Let $V$ be an absolutely $p$-divisible $kG$-module, let $C \in Ch^b({}_{kG}\triv)$, and let $\calX_V$ be the set of $p$-subgroups $P$ of $G$ for which $V(P) = 0$. The following are equivalent:

        \begin{enumerate}
            \item $C$ is $V$-endosplit-trivial.
            \item $C(P)$ has nonzero homology concentrated in exactly one degree for all $P \in s_p(G)$, and that homology has $k$-dimension one when $P \in \calX_V$.
            \item $C(P)$ has nonzero homology concentrated in exactly one degree for all $P \in s_p(G)$, and for the unique $i\in \Z$ satisfying $H_i(C) \neq 0$, $H_i(C)$ is a $V$-endotrivial $kG$-module.
        \end{enumerate}

    \end{corollary}
    \begin{proof}
        (a) and (c) are equivalent by Proposition \ref{prop:equivendosplit} and the fact that $V$-endosplit-trivial complexes are equivalently shifted endosplit $p$-permutation resolutions of relatively $V$-endotrivial modules. (c) implies (b) is a straightforward verification. If (b) holds, then $C$ is an endosplit $p$-permutation resolution by Proposition \ref{prop:equivendosplit} and weakly $V$-endotrivial by Theorem \ref{equivdefinitionweakendotrivial}, and it follows that $C$ is $V$-endosplit-trivial.
    \end{proof}

    \begin{remark}
        As before, it follows by commutativity with tensor products that applying restriction, inflation, or the Brauer construction to an endosplit $p$-permutation resolution will result in an endosplit $p$-permutation resolution. We will see in the sequel that if $H \leq G$, applying additive induction to an endosplit $p$-permutation resolutions which satisfies a stability condition results in an endosplit $p$-permutation resolution. However, this will not in general preserve $V$-endosplit-triviality.
    \end{remark}

    The following result of Rickard in \cite{R96} will be of use in the sequel. We state a version provided in \cite{L182}.

    \begin{theorem}{\cite[Theorem 7.11.2]{L182}} \label{endosplitsummands}
        Let $M$ be a finitely generated $kG$-module having an endosplit $p$-permutation resolution $C$. For any subgroup $H \leq G$, the complex $\Hom_{kG}(C,C)$ has homology concentrated in degree zero, isomorphic to $\Hom_{kH}(M,M)$, and there is an algebra isomorphism \[\rho_H: \Hom_{K^b({}_{kH}\catmod)}(C,C)\cong \Hom_{kH}(M,M)\] satisfying $\Res^G_H \circ\rho_G = \rho_H \circ \Res^G_H$ and $\tr^G_H \circ \rho_H = \rho_G \circ \tr^G_H$. In particular, $\rho_G$ induces a vertex and multiplicity preserving bijection between the sets of isomorphism classes of indecomposable direct summands of the module $M$ and of noncontractible indecomposable direct summands of the complex $C$.
    \end{theorem}

    As a partial converse, we can determine when the direct sum of two endosplit $p$-permutation resolutions is an endosplit $p$-permutation resolution.

    \begin{theorem}\label{sumsofendosplitpperms}
        Let $C, D$ be two endosplit $p$-permutation resolutions. $C \oplus D$ is an endosplit $p$-permutation resolution if and only if for all $P \in s_p(G)$, one of the two cases holds: both $C(P)$ and $D(P)$ are contractible or $C(P)$ and $D(P)$ have nonzero homology concentrated in the same degree.
    \end{theorem}
    \begin{proof}
        First, suppose that for all $P\in s_p(G)$, either both $C(P)$ and $D(P)$ are contractible or $C(P)$ and $D(P)$ have nonzero homology concentrated in the same exact degree. Then the same holds for $(C\oplus D)(P)$, so Proposition \ref{prop:equivendosplit} implies $C \oplus D$ is an endosplit $p$-permutation resolution, as desired. Now suppose $C \oplus D$ is an endosplit $p$-permutation resolution. Then $(C \oplus D)(P) \cong C(P) \oplus D(P)$ is either contractible or has nonzero homology in exactly one degree for any $P \in s_p(G)$. Therefore, both $C(P)$ and $D(P)$ are both contractible or have nonzero homology in the same unique degree as well. The result follows.
    \end{proof}

    \begin{remark}
        If the hypotheses in the previous theorem occur, then the endosplit $p$-permutation resolution $C \otimes_k D^*$ has h-marks entirely 0. By a similar argument as in Proposition \ref{prop:equivendosplit}, $C\otimes_k D^* \simeq (M\otimes_k N^*)[0]$, where $M$ and $N$ are the $kG$-modules which $C$ and $D$ respectively resolve. This implies $M\otimes_k N^*$ is a $p$-permutation module, so the above situation can only occur if $M\otimes_k N^*$ is $p$-permutation. If $G$ is a $p$-group and $C,D$ are both $V$-endosplit-trivial complexes, this implies $M \cong N$ modulo $V$-projectives, however, for an arbitrary finite group $G$, this isomorphism does not hold.

        In particular, setting $N := M$ and $D := C$ implies that the only $kG$-modules which can possibly have endosplit $p$-permutation resolutions are \textit{endo-$p$-permutation} $kG$-modules, that is, $kG$-modules $M$ for which $M^* \otimes_k M \cong k \oplus P$, where $P$ is a $p$-permutation module. Endo-$p$-permutation modules were first studied in detail by Urfer in \cite{Ur06}.
    \end{remark}

    \section{h-mark homomorphisms}

    In \cite{SKM23}, we observed that much of the structure of an endotrivial complex could be determined by local homological data, which we coined ``h-marks.'' We can apply the same techniques in the relative endotriviality setting. In this section, we construct h-mark homomorphisms which determine much of the structure of a relatively endotrivial complex, particularly for weakly $V$-endotrivial complexes and $V$-endosplit-trivial complexes. We continue to assume $V$ is an absolutely $p$-divisible $kG$-module, and set $\calX_V \subseteq s_p(G)$ to be the subset of all $p$-subgroups for which $V(P) = 0$.

    \textbf{Notation:} We denote the set of $\Z$-valued superclass functions on $G$, that is, maps from the set of subgroups of $G$ to $\Z$ constant on conjugacy classes, by $CF(G)$. If $\calX$ is a subset of the set of subgroups of $G$ which is closed under $G$-conjugation, then $CF(G, \calX)$ denotes the group of \textit{$\Z$-valued superclass functions on $\calX$}, i.e. the group of functions $f: \calX \to \Z$ which are constant on conjugacy classes. We set $CF(G, p) := CF(G, s_p(G))$.

    \subsection{h-mark homomorphisms for weakly $V$-endotrivial complexes}

    The definitions in this section arise from the fact that if $C$ is a weakly $V$-endotrivial complex and $P \in \calX_V$, then $C(P)$ is an endotrivial complex since $V(P) = 0$ (see Definition \ref{def:reletrivdef}).

    \begin{definition}
        Let $V$ be an absolutely $p$-divisible $kG$-module. Recall that for any $P \in \calX_V$, $C(P)$ is an endotrivial complex.

        For a class of weakly $V$-endotrivial complexes $[C] \in w\calE_k^V(G)$ and $p$-subgroup $P$ for which $V(P) = 0$, denote by $h_C(P)$ the degree $i\in \Z$ of a representative $C$ in which $H_i(C(P)) \neq 0$, and if we have $H_{h_C(P)}(C(P)) \cong k_\omega$, set $\calH_C(P):= \omega \in \Hom(N_G(P)/P,k^\times)$. This is well-defined since $C(P)$ is an endotrivial complex. We call the values of $h_C(P)$ the \textit{h-marks of $C$ at $P$}, and abusively refer to $\calH_C(P)$ as \textit{the homology of $C$ at $P$.}

        We have a map:
        \begin{align*}
            \Xi: w\calE^V_k(G) &\to \prod_{P \in \calX_V} \Z \times \Hom(N_G(P)/P, k^\times)\\
            [C] &\mapsto \big(h_C(P), \calH_C(P)\big)_{P \in s_p(G)}
        \end{align*}

        It is straightforward to verify that this map is a well-defined group homomorphism via the K{\"u}nneth formula and commutativity of the Brauer construction with tensor products for $p$-permutation modules.

        For any weakly $V$-endotrivial complex $C$, we may regard $h_C$ as a $\Z$-valued superclass function on $\calX_V$. Indeed, the Brauer construction satisfies $C({}^gP) = {}^gC(P)$ for any $g \in G$ and $P \in s_p(G)$. This gives another group homomorphism, which we call the \textit{h-mark homomorphism,}

        \begin{align*}
            h_w: w\calE^V_k(G) &\to CF(G,\calX_V)\\
            [C] &\mapsto h_C
        \end{align*}

        One may easily verify that this is well-defined, that is, if $[C_1] = [C_2] \in w\calE_k^V(G)$, the $h_{C_1} = h_{C_2} \in  CF(G,\calX_V)$.
    \end{definition}

    \begin{lemma}{\cite[Lemma 5.8]{SKM23}}\label{relprojectivitylimitingthm}
        Let $C$ be a bounded complex of $p$-permutation $kG$-modules with the following property: there exists a set $\calY \subset s_p(G)$ which does not contain Sylow subgroups, is closed under conjugation and taking subgroups, and that satisfies that there exists some $i \in \Z$ such that for all $P \not\in \calY$, $\dim_k H_i(C(P)) = 1$ and $C(P)$ is exact in all other degrees. Then $C$ is homotopy equivalent to an complex $C'$ with the following property: there exists an $i \in \Z$ such that $C'_i = M \oplus N$, with $M$ $\calY$-projective, possibly 0, and a vertex of $N$ is a Sylow $p$-subgroup, and for all $j \neq i$, either $C_j' = 0$ or $C_j'$ is $\calY$-projective.
    \end{lemma}

    \begin{theorem}\label{thm:kerofweakhmk}
        We have $\ker (h_w) \cong T_{k,V}(G,S).$ In particular, $\ker (h_w)$ is the torsion subgroup of $w\calE_k^V(G)$, and $w\calE_k^V(G)$ is a finitely generated abelian group.
    \end{theorem}
    \begin{proof}
        We have $T_{k,V}(G,S) \leq \ker (h_w)$, so it suffices to show the converse inclusion. The elements of $\ker (h_w)$ are (isomorphism classes of) $V$-endotrivial complexes $C$ for which, given any $P \in \calX_V$, $\dim_k H_0(C(P)) = 1$ and $C(P)$ is exact in all other degrees. Any such complex $C$ satisfies the hypothesis of the previous lemma, where we set $\calY = s_p(G) \setminus \calX_V$ in the notation of the previous lemma.

        For weakly $V$-endotrivial complexes, $V$-projectivity is the same as $\calY$-projectivity, by Proposition \ref{replacewithppermformodules} and \cite[Proposition 5.10.3]{L181}. Thus, $C_0$ has an indecomposable summand with vertex $S$, and all other indecomposable modules appearing in each degree of $C$ are $V$-projective. Therefore, $C \cong M[0]$ in $K^b({}_{kG}\sttriv_V)$ for some indecomposable trivial source $V$-endotrivial $kG$-module $M$ with vertex $S$, thus $\ker(h_W)\subseteq T_{k,V}(G,S)$ as desired.

        It follows that $\ker(h_w)$ is the torsion subgroup of $w\calE_k^V(G)$, since $T_{k,V}(G,S)$ is finite (as there is a finite number of isomorphism classes trivial source $kG$-modules), hence torsion, and any complex with a nonzero h-mark cannot be torsion, as successive powers of the complex will increase the absolute value of the nonzero h-mark. Additionally, $w\calE_k^V(G)$ must be finitely generated abelian, since the $\Z$-rank of $w\calE_k^V(G)$ is bounded below by the $\Z$-rank of $CF(G,\calX_V)$, which is equal to the number of conjugacy classes of $\calX_V$.
    \end{proof}

    \begin{remark}
        In particular, since $\ker (h_w)$ is the torsion subgroup of $w\calE^V_k(G)$, we have a split exact sequence of abelian groups \[0 \to T_{k,V}(G,S) \to w\calE_k^V(G) \to \im ( h_w) \to 0,\] since $\im(h_w)$ is a free abelian group. Unlike the situation for endotrivial complexes given in \cite[Remark 3.8]{SKM23}, this sequence is not canonically split.
    \end{remark}

    \begin{corollary}
        Two weakly $V$-endotrivial complexes are isomorphic in $K^b({}_{kG}\sttriv_V)$ if and only if they have the same h-marks and local homology at all $P \in \calX_V$.
    \end{corollary}

    \begin{proof}
        This follows since given two weakly $V$-endotrivial complexes which satisfy the above, they must represent the same class in $w\calE_k^V(G)$, whence the result.
    \end{proof}

    \subsection{h-marks for endosplit $p$-permutation resolutions and $V$-endosplit-trivial complexes}

    Recall that a complex of $p$-permutation $kG$-modules is a shifted endosplit $p$-permutation resolution if and only if for all $P \in s_p(G)$, $C(P)$ either has nonzero homology in exactly one degree or is contractible.

    \begin{definition}
        Let $C$ be either a $V$-endosplit-trivial complex or an endosplit $p$-permutation resolution, and let $P \in s_p(G)$. Denote by $h_C(P)$ the unique degree $i \in \Z$ for which $H_i(C(P)) \neq 0$.

        If $C$ is a $V$-endosplit-trivial complex, set $\calH_C(P) = [H_{h_C(P)}(C(P))] \in T_{k,V}(N_G(P)/P)$. We claim this is invariant among equivalent complexes in $e\calE_k^V(G)$. Indeed, suppose $[C_1] = [C_2] \in e\calE_k^V(G)$. Then $C_1$ and $C_2$ have isomorphic indecomposable $V$-endosplit-trivial direct summands, which we denote $C_0$. Therefore we may write $C_1 \cong C_0 \oplus P_1$ and $C_2 \cong C_0 \oplus P_2$, with $P_1, P_2$ homotopy equivalent to $V$-projective complexes. Therefore, for all $i \in \Z$, $H_i(P_1)$ and $H_i(P_2)$ are $V$-projective $kG$-modules. Moreover, for any $P \in s_p(G)$, $h_{C_1}(P) = h_{C_0}(P) = h_{C_2}(P)$. Thus, for any $P \in s_p(G)$, we have

        \begin{align*}
            \calH_{C_1}(P) &= [H_{h_{C_1}(P)}(C_0(P)) \oplus H_{h_{C_1}(P)}(P_1(P))] \\
            &= [H_{h_{C_0}(P)}(C_0(P))]  \\
            &= [H_{h_{C_2}(P)}(C_0(P)) \oplus H_{h_{C_2}(P)}(P_2(P))] \\
            &= \calH_{C_2}(P) \in T_{k,V}(N_G(P)/P).
        \end{align*}

        As before, we call the values of $h_C(P)$ the \textit{h-marks of $C$ at $P$}, and abusively refer to $\calH_C(P)$ as \textit{the homology of $C$ at $P$}. Note that here, we are taking the image of homology in $T_{k,V}(G)$ rather than the isomorphic image. We obtain a group homomorphism:

        \begin{align*}
            \Xi: e\calE_k^V(G) &\to \prod_{P \in s_p(G)} \Z \times T_{k, V(P)}(N_G(P)/P)\\
            [C] &\mapsto (h_C(P), \calH_C(P))_{P \in s_p(G)}.
        \end{align*}

        We identify $T_{k,\{0\}}(G)$ with $\Hom(G,k^\times)$, the group of $k$-dimension one $kG$-modules with addition given by $\otimes_k$. We regard $h_C$ as a $\Z$-valued superclass function on $s_p(G)$, as before. This gives another group homomorphism, which we call the \textit{h-mark homomorphism},

        \begin{align*}
            h_e: e\calE_k^V(G) &\to CF(G,p)\\
            [C] &\mapsto h_C
        \end{align*}

        Note that the codomain of $h_e$ is $CF(G,p)$, whereas the codomain of $h_W$ is $CF(G,\calX_V)$.

    \end{definition}

    \begin{theorem}
        We have $\ker (h_e) \cong T_{k,V}(G,S)$. In particular, $\ker (h_e)$ is the torsion subgroup of $e\calE_k^V(G)$, and $e\calE_k^V(G)$ is finitely generated.
    \end{theorem}
    \begin{proof}
        Note $T_{k,V}(G,S) \leq \ker (h_e)$. Suppose $[C] \in \ker (h_e)$, and let $C$ be an indecomposable representative of $[C]$. First suppose $C$ has exactly one nonzero term - since $[C] \in \ker (h_e)$, it follows that the nonzero term must be in degree 0. It follows that $C_0^* \otimes_k C_0 \cong k \oplus P$, for some $V$-projective $kG$-module $P$, i.e. $C_0$ is a relatively $V$-endotrivial module. Moreover, $C_0$ is $p$-permutation and therefore has a unique trivial source $V$-endotrivial direct summand, whence the result.

        Now, suppose for contradiction that $C$ has at least 2 nonzero terms. Then one term of $C$ must be in a nonzero degree. Suppose without loss of generality $C$ has a nonzero terms in a positive degree. Let $i > 0$ be the greatest integer for which $C_i \neq 0$. We label the differentials of $C$ by $d_i$. Since $C$ is exact at $i$, $d_i: C_i \to C_{i-1}$ is injective. Moreover, for any $P \in s_p(G)$, since $C(P)$ is exact at $i$ and $C_j(P) = 0$ for all $j > i$, $d_i(P)$ is injective as well. Therefore, $d_i$ is split injective by \cite[Theorem 5.8.10]{L181}. This implies a contractible term isomorphic to $0 \to C_i \xrightarrow{\id} C_i \to 0$ splits off $C$, however this is a contradiction since we assumed $C$ to be indecomposable. Therefore, the only elements of $\ker(h_e)$ are of the form $[M[0]]$, where $M$ is a trivial source $V$-endotrivial $kG$-module, as desired. The final statements follow similarly to those in Theorem \ref{thm:kerofweakhmk}.
    \end{proof}

    \begin{remark}
        In particular, since $\ker( h_e)$ is the torsion subgroup of $e\calE^V_k(G)$, we have a split exact sequence of abelian groups \[0 \to T_{k,V}(G,S) \to e\calE_k^V(G) \to \im (h_e) \to 0,\] since $\im (h_e)$ is a free abelian group. Unlike the situation for endotrivial complexes, this sequence is not canonically split, one must choose a basis for the free part of $e\calE_k^V(G)$. The reason for this is that in general, given any $C \in e\calE^V_K(V)$, $\calH_C(P)\in T_{k,V}(N_G(P)/P)$ may not have a $p$-permutation representative, and therefore we do not have a clear well-defined map from $e\calE_k^V(G)$ to $T_{k,V}(G,S)$.
    \end{remark}

    \begin{corollary}
        Two indecomposable $V$-endosplit-trivial chain complexes are isomorphic chain complexes if and only if they have the same h-marks and isomorphic nonzero homology.
    \end{corollary}

    \section{Borel-Smith conditions}

    In this section, we give a partial set of conditions which h-mark functions must satisfy. These conditions come from Borel-Smith functions, a type of superclass function for $p$-groups which arise both from homotopy representations of spheres (see \cite[Page 210]{td87}) and from the study of endo-permutation modules (see \cite{BoYa06}). Borel-Smith functions are defined by conditions imposed at certain subquotients of a group. We will show that these same conditions must be met for our h-mark functions, but only at certain subquotients which satisfy a condition relating to $V$. To do this, we will show that h-mark functions of endotrivial complexes are Borel-Smith functions.

    \begin{definition}\label{def:borelsmith}{\cite[Definition 3.1]{BoYa06}}
        A superclass function $f \in CF(G)$ is called a \textit{Borel-Smith function} if it satisfies the following conditions:
        \begin{enumerate}
            \item If $p$ is odd, $H \trianglelefteq L \leq G,$ and $L/H \cong \Z/p\Z$, then $f(H) - f(L)$ is even.
            \item If $H \trianglelefteq L \trianglelefteq N \leq N_G(H)$ and $L/H \cong \Z/2\Z$, then $f(H) - f(L)$ is even if $N/H\cong \Z/4\Z$, and $f(H) - f(L)$ is divisible by $4$ if $N/H$ is the quaternion group of order 8.
            \item If $H \trianglelefteq L \leq G, L/H \cong \Z/p\Z\times \Z/p\Z$, $H_i/H$ the subgroups of order $p$ in $L/H$, then \[f(H) - f(L) = \sum_{i=0}^p \big(f(H_i) - f(L)\big).\]
        \end{enumerate}

        These conditions are referred to as the \textit{Borel-Smith conditions}. The were first discovered as the conditions satisfied by the dimension function of a homology mod $p$ sphere with a $G$-action. The set of Borel-Smith functions is an additive subgroup of $CF(G)$ which we denote by $CF_b(G)$. Note that Condition (b) sometimes requires that $f(H) - f(L)$ is divisible by $4$ when $N/H$ is a generalized quaternion 2-group. However, this requirement is redundant, since any generalized quaternion group of order at least 8 contains the quaternion group of order 8 as a subgroup, which contains its unique subgroup of order 2.

        Given any set $\calX$ of subgroups of $G$ that is closed under $G$-conjugation, we say $f \in CF(G, \calX)$ is a Borel-Smith function if it satisfies the above conditions wherever possible. For instance, if $H \trianglelefteq L \trianglelefteq N \leq N_G(H)$ satisfies $L/H \cong \Z/2\Z$ and $N/H \cong Q_8$, if $H, L \in \calX$ but $N \not\in \calX$, $f(H) - f(L)$ must be divisible by $4$, since $f(H)$ and $f(L)$ are defined, even though $f(N)$ is not. Again, the set of Borel-Smith functions on $\calX$ is an additive subgroup of $CF(G,\calX)$ which we denote by $CF_b(G, \calX). $
    \end{definition}

    We first show the h-marks of an endotrivial complex must satisfy the Borel-Smith conditions.

    \begin{lemma}
        \begin{enumerate}
            \item Let $p$ be an odd prime and $G = \Z/p\Z$. If $C$ is an endotrivial complex of $kG$-modules, then $h_C(1) - h_C(G)$ is even.
            \item Let $p = 2$ and $G = \Z/4\Z$. Let $H$ be the unique subgroup of order 2 in $G$. If $C$ is an endotrivial complex of $kG$-modules, then $h_C(1) - h_C(H)$ is even.
            \item Let $p = 2$ and $G = Q_8$. Let $Z$ be the unique subgroup of order 2 in $G$. If $C$ is an endotrivial complex of $kG$-modules, then $h_C(1) - h_C(Z)$ is divisible by 4.
            \item Let $G = \Z/p\Z \times \Z/p\Z$, and let $H_0, \dots, H_p$ be the subgroups of order $p$ in $G$. If $C$ is an endotrivial complex of $kG$-modules, then \[h_C(1) - h_C(G) = \sum_{i=0}^p h_C(H_i) - h_C(G).\]
        \end{enumerate}
    \end{lemma}
    \begin{proof}
        (a) through (c) may be shown directly by the classifications of $\calE_k(G)$ given in \cite[Section 6]{SKM23}, however we provide a direct proof. Let $C$ be an endotrivial complex of $kG$-modules. It follows from \cite[Proposition 4.3 (b) and (c)]{SKM23} that $\hat{C} := \Inf^G_{G/Z} C(Z)$ is an endotrivial complex of $kG$-modules which satisfies that for all $1 \neq H \leq G$, $h_C(H) = h_{\hat{C}}(H)$, and $h_{\hat{C}}(1) = h_C(Z).$ Therefore, $D:=C \otimes_k \hat{C}^*$ is an endotrivial complex satisfying for all $1 \neq H \leq G$, $h_D(H) = 0$ and $h_D(1) = h_C(1) - h_C(Z)$.

        Lemma \ref{relprojectivitylimitingthm} implies that $D$ is homotopy equivalent to an endotrivial chain complex with the trivial module in degree 0 and the only other $kG$-modules which appear in $D$ are projective. Moreover, $D$ is acyclic in only one degree, $h_D(1) = h_C(1) - h_C(Z)$, and since it is endotrivial, that homology has $k$-dimension 1, and therefore must be the trivial $kG$-module, since $G$ is a $p$-group for (a)-(c). It follows that $D$ contains as a direct summand a truncated periodic resolution or injective hull of the trivial $kG$-module. The result now follows by considering the period the trivial $kG$-module in (a)-(c): in (a) and (b) $k$ has period 2 and in (c) $k$ has period 4.

        (d) follows directly from the classification of $\calE_k(\Z/p\Z \times \Z/p\Z)$ given in \cite[Proposition 6.4]{SKM23}, since every element in the basis of $\calE_k(\Z/p\Z \times \Z/p\Z)$ given in the proposition satisfies the condition.

    \end{proof}

    \begin{theorem}
        Let $C$ be an endotrivial complex. Then the h-marks of $C$ satisfy the Borel-Smith conditions, i.e. $h_C \in CF_b(G,p)$.
    \end{theorem}
    \begin{proof}
        We will show condition (a) is satisfied, the other conditions follow by analogous arguments. Suppose $p$ is odd, $H \trianglelefteq L \leq G$ with $H, L \in s_p(G)$, and $L/H \cong \Z/p\Z$. Then $D:=\Res^{N_G(H)/H}_{L/H} C(H)$ is an endotrivial complex of $k[L/H]$-modules, and \cite[Proposition 4.3 (a) and (b)]{SKM23} imply $h_D(L/H) = h_C(L)$, $h_D(H/H) = h_C(H)$. The previous lemma implies $h_D(H/H) - h_D(L/H)$ is even, thus $h_C(H) - h_C(L)$ is even as well, as desired.
    \end{proof}

    \begin{corollary}
        Let $V$ be an absolutely $p$-divisible $p$-permutation $kG$-module, let $\calX_V \subseteq s_p(G)$ denote the set of $p$-subgroups of $G$ for which $V(P) = 0$, and let $C$ be a weakly $V$-endotrivial complex of $kG$-modules. Then the h-marks of $C$ satisfy the Borel-Smith conditions, i.e. $h_C \in CF_b(G, \calX_V)$.
    \end{corollary}
    \begin{proof}
        This follows since for every $P \in \calX_V,$ $C(P)$ is an endotrivial complex, and one may verify an analogue of \cite[Proposition 4.3(b)]{SKM23}: for any pair of $p$-subgroups $P\trianglelefteq Q \in \calX_V$, we have $h_C(Q) = h_{C(P)}(Q/P)$.
    \end{proof}

    \begin{definition}
        Let $V$ be a $p$-permutation $kG$-module. Say $f \in CF_b(G,p)$ \textit{satisfies the Borel-Smith conditions at $V$} if $f$ satisfies the conditions given in Definition \ref{def:borelsmith} whenever the $p$-subgroup $H$ (in the notation of the conditions) satisfies $V(H) = 0$.
    \end{definition}

    \begin{corollary}
        Let $V$ be an absolutely $p$-divisible $p$-permutation $kG$-module, and let $C$ be a $V$-endosplit-trivial complex of $kG$-modules. Then the h-marks of $C$ satisfy the Borel-Smith conditions at $V$.
    \end{corollary}
    \begin{proof}
        Again, this follows since for every $P$ satisfying $V(P) = 0$, $C(P)$ is an endotrivial complex.
    \end{proof}

    In future work, we will use the machinery developed in this paper to prove that, conversely, every Borel-Smith function $f \in CF_b(G,p)$ is the h-mark function of an endotrivial complex. The following conjectures are generalizations of this fact.

    \begin{conjecture}\label{conj}
        Let $V$ be an absolutely $p$-divisible $p$-permutation $kG$-module.
        \begin{itemize}
            \item If $f \in CF_b(G, \calX_V)$, then there exists a weakly $V$-endotrivial complex $C$ such that $f = h_C$.
            \item If $f \in CF(G,p)$ satisfies the Borel-Smith conditions at $V$, then there exists a $V$-endosplit-trivial chain complex $C$ such that $f = h_C$.
            \item Let $C$ be a weakly $V$-endotrivial complex. The isomorphism class $[C] \in w\calE_k^V(G)$ contains a $V$-endosplit-trivial complex.
        \end{itemize}
        
    \end{conjecture}
    If the first two assertions of Conjecture \ref{conj} hold, the third does as well.

    \section{The case of $V = kG$}

    In this section, we examine the case of $V = kG$. We set $V := kG$ for the entire section for ease of notation. Here, $V$-endotrivial modules are endotrivial modules in the usual sense. Therefore, $V$-endosplit-trivial complexes are shifted endosplit $p$-permutation resolutions of (non-relative) endotrivial modules. In fact, modulo $V$-projective chain complexes, strongly $V$-endotrivial complexes are $V$-endosplit-trivial complexes.

    \begin{theorem}
        Let $C$ be an indecomposable strongly $V$-endotrivial complex. Then $C$ is an $V$-endosplit-trivial complex. In particular, the inclusion $e\calE_k^{kG}(kG) \hookrightarrow s\calE_k^{kG}(kG)$ is an equality.
    \end{theorem}
    \begin{proof}
        First, note that $V$-projective chain complexes of $kG$-modules can be expressed as direct sums of shifts of chain complexes of the form $0\to P \xrightarrow[]{\sim} P \to 0$ or $0\to P \to 0$, where $P$ is a projective $kG$-module. This follows since all $k$-vector space chain complexes decompose in this way. Therefore, $C^* \otimes_k C \simeq k[0] \oplus D$, where $D$ is a $V$-projective chain complex, and we may assume without loss of generality that it has no contractible components. Therefore, $D$ is a chain complex with zero differentials and projective homology. Moreover, since $C^* \otimes_k  C \cong C \otimes_k C^*$, it follows that $D^* \cong D$, and since $D$ has zero differentials, equivalently $D_i^* \cong D_{-i}$. It suffices to show that $H_i(D) = 0$ for all $i \neq 0$.

        We have by the K\"unneth formula and exactness of the $k$-dual that \[H_0(C^*\otimes_k C) \cong \bigoplus_{i\in \N_{\geq 0}} H_{-i}(C^*) \otimes_k H_{i}(C) \cong \bigoplus_{i\in \N_{\geq 0}} H_i(C)^* \otimes_k H_i(C),\] where the direct sum has a finite number of nonzero terms. By the characterization of $D$, $H_0(C^*\otimes_k C) \cong k \oplus P$ for some projective $kG$-module $P$. Therefore, by the Krull-Schmidt theorem there exists some $i \in \N_{\geq 0}$ for which $H_i(C)$ is endotrivial, and for all $j \neq i$, for which $H_j(C) \neq 0$, $H_j(C) \otimes_k H_j(C)^*$ is projective. From \cite[Proposition 2.2.2]{CL12}, it follows that $H_j(C)$ is projective, since projectivity is the same as $V$-projectivity for modules.

        Suppose for contradiction there exists some $j$ for which $H_j(C)$ is projective. Then we have a short exact sequence of $kG$-modules \[0 \to \im (d_{j+1}) \to \ker (d_j) \to H_j(C) \to 0,\] and projectivity implies this sequence is split. Therefore $H_j(C)$ is a submodule of $\ker (d_j)$, hence a submodule of $C_j$. Since $kG$ is self-injective, $H_j(C)$ is injective as well, hence $H_j(C)$ is in fact a module direct summand of $C_j$. Now since $\ker (d_j) = H_j(C) \oplus \im (d_{j+1})$, in particular $H_j(C) \cap \im (d_{j+1}) = \{0\}$ so $C$ has as direct summand the singleton chain complex $0\to H_j(C) \to 0$, contradicting that $C$ is indecomposable.

        Therefore, there can not exist any integers $j$ for which $H_j(C)$ is projective. It follows that $H_i(C)$ is nonzero for exactly one $i \in \Z$, hence $C$ is an endosplit $p$-permutation resolution of an endotrivial module, i.e. a $V$-endosplit-trivial complex.
    \end{proof}

    Furthermore, every weakly $V$-endotrivial complex is isomorphic in $K^b({}_{kG}\sttriv_V)$ to a $V$-endosplit-trivial complex, as we now demonstrate.

    \begin{construction}
        Given any bounded chain complex $C$, one can construct a new chain complex $D$ that has nonzero homology in at most one degree $d$ as follows. Suppose without loss of generality $C_i = 0$ for all $i < 0$ and $i > n$, and for ease of explanation, suppose $0 < d < n$. Start with $D = C$, then add a projective cover $P_1$ of $C_0$ to $D_1$ as direct sum and add the corresponding cover $P_1 \to D_0$ to the differential $d_1$. Then $H_0(D) = 0$. Then, if $1 \neq d$, add a projective cover $P_2$ of $\ker (d_1)$ to $D_2$ and add the corresponding cover $P_2 \to \ker (d_1) \hookrightarrow D_1$. By construction, $H_1(D) = 0$, and $d_0 \circ d_1 = 0$. Repeat this process at all $D_i$ for $i \{0,\dots, d-1\}$. Then, $H_i(D) = 0$ for all $i < d$.  Visually, the construction is as follows:
        \begin{figure}[H]
            \centering
            \begin{tikzcd}
                C_d \ar[r]& C_{d-1} \ar[r] & \cdots \ar[r] & C_1 \ar[r] & C_0\\
                P_d \ar[ru]\ar[r] & P_{d-1}\ar[ru] \ar[r] & \cdots \ar[ru] \ar[r] & P_1 \ar[ru] \\
            \end{tikzcd}
        \end{figure}

        Then, take an injective hull of $C_n,$ labeled $I_{n-1}$, and add it to $D_{n-1}$ as direct sum and add the envelope to the differential $d_n$. Then $H_n(D) = 0$. Then if $n-1 \neq d$, let $I_{n-2}$ be an injective envelope of $D_{n-1}/\im (d_n)$. Add $I_{n-2}$ to $D_{n-2}$, and add the corresponding composition $D_{n-1} \rightarrow D_{n-1}/\im (d_n) \hookrightarrow I_{n-2}$ to $d_{n-1}$. By construction, $d_{n-1}\circ d_n = 0$ and $H_{n-1}(D) = 0$. Repeat this process for all $i \in \{n, n-1,\dots, d+1\}$. Then $H_i(D) = 0$ for all $i > d$. Visually, the resulting complex $D$ is as follows:
        \begin{figure}[H]
            \centering
            \begin{tikzcd}
                C_n \ar[r] \ar[rd] & C_{n-1} \ar[r] \ar[rd] & \cdots \ar[r] \ar[rd] & C_{d+1} \ar[r] \ar[rd] & C_d \ar[r]& C_{d-1} \ar[r] & \cdots \ar[r] & C_1 \ar[r] & C_0\\
                & I_{n-1} \ar[r] & \cdots \ar[r] & I_{d+1} \ar[r]  & I_d \oplus P_d \ar[ru]\ar[r] & P_{d-1}\ar[ru] \ar[r] & \cdots \ar[ru] \ar[r] & P_1 \ar[ru] \\
            \end{tikzcd}
        \end{figure}

        Note that $C(P) = D(P)$ for any nontrivial $p$-subgroup $P \in s_p(G)$. In particular, $C \cong D$ in $K^b({}_{kG}\stmod_V)$.
    \end{construction}

    \begin{remark}
        It may be possible to, in some cases, perform a local analogue of this process as well. Suppose for instance that we wish to augment a chain complex of $kG$-modules $C$ such that $C(P)$ has homology concentrated in one degree. The Green correspondence gives a bijection between projective indecomposable $k[N_G(P)/P]$-modules and trivial source $kG$-modules with vertex $P$. Therefore, we can append trivial source $kG$-modules with vertex $P$ to $C$ in such a way that the additions correspond to taking projective resolutions or injective hulls of the $P$-local kernels and cokernels. The difficulty which arises, however, is that while it is possible to draw the differentials at the $P$-local level to be projective resolutions and injective hulls of differentials of $C(P)$, there is no guarantee that the corresponding maps at the global level result are still differentials, i.e. we may have $d^2 \neq 0$.

        If one was able to show that this process can be carried out without issues, then it would follow that given a weakly $V$-endotrivial complex $C$, the equivalence class $[C] \in w\calE_k^V(G)$ contains a $V$-endosplit-trivial complex.
    \end{remark}

    \begin{theorem}
        The group homomorphism $\iota: e\calE_k^V(G) \to w\calE_k^V(G)$ is surjective. Moreover, $\ker(\iota)$ is generated by representatives of truncated projective resolutions or injective hulls of the trivial $kG$-module $k$.
    \end{theorem}
    \begin{proof}
        The surjectivity of $\iota$ follows immediately from the previous construction. Suppose $C$ is a $V$-endosplit-trivial complex of $kG$-modules with $[C] \in w\calE_k^V(G)$. Since the h-marks of $C$ are trivial, it follows by Lemma \ref{relprojectivitylimitingthm} that $C$ is homotopy equivalent to a complex consisting of only projectives, with the exception of one indecomposable module in degree zero which is endotrivial. Moreover, since $[C] \in \ker(\iota)$, this complex is stably isomorphic to the trivial endotrivial complex, so the indecomposable endotrivial complex is the trivial complex $k[0]$. It follows that $C$ is homotopy equivalent to a complex of projectives in all degrees except 0, which contains only the trivial $kG$-module, and has nonzero homology in only 1 degree since $C$ is $V$-endosplit-trivial. A standard homological algebra argument implies that $C$ is homotopy equivalent to a truncated projective resolution, as desired.
    \end{proof}

    \begin{example}
        \begin{enumerate}
            \item When $G = SD_{2^n}$ is a semidihedral 2-group of order at least 16, recall from \cite{SKM23} that there exists a faithful endotrivial complex $C \in \calE_k(G)$ for which $h_C(H) = 2$ for $H\leq G$, the unique (up to conjugacy) noncentral subgroup of order 2, $h_C(1) = 4$, and $h_C(K) =0$ when $K \neq 1,H $. Define the $G$-set $X := G/H$. Recall (see for instance \cite{CaTh00}) that $E := \Omega(\Omega_X(k))$ is an endotrivial $kG$-module which satisfies $[E]^2 = [k] \in T(G)$. From this, it follows that \[C_E := kG \to kX \to k,\] with $k$ in degree 0, the differential $kX \to k$ the augmentation map, and the differential $kG \to kX$ corresponding to the projective cover of $\Omega_X(k)$, is a $V$-endotrivial complex with $h_{C_E}(1) = 2$, $h_{C_E}(H) = 1$, and $h_{C_E}(K) = 0$ for $K \neq 1, H$. Therefore, $[C_E\otimes_k C_E] = [C] \in \calE_k^1(G)$, that is, they are isomorphic up to homotopy equivalence and a direct sum of complexes of the form $0 \to kG\to 0$. We conclude it is possible for there to be torsion in the quotient $\calE_k^1(G)/\calE_k(G)$.

            \item Given any group $G$ which has $p$-rank greater than 1, the chain complex $C = kG \to k$ has $H_1(C) \cong \Omega(k)$, which is an endotrivial module. This cannot be an endotrivial complex since $\dim_k \Omega(k) = |G|- 1$, but is $V$-endotrivial. Since $G$ has $p$-rank greater than 1, for all $i \geq 2$, it follows that $C^{\otimes i}$ will also be $V$-endotrivial but not endotrivial, thus $\rk_\Z \calE_k^1(G)$ will almost always be strictly greater than $\rk_\Z \calE_k(G)$.
        \end{enumerate}

    \end{example}

    \section{Induction and restriction for subgroups containing Sylow $p$-subgroups} \label{sylowsection}

    In the study of endotrivial modules, a common goal is to determine the image of the restriction map $\Res^G_H: T_k(G) \to T_k(H)$, where $H \leq G$ contains a Sylow $p$-subgroup. In this section, we will consider the analogous question for endotrivial and relatively endotrivial complexes. To do this, we will consider induction of these complexes as well. We will show that as long as a $G$-stability condition is satisfied, induction preserves the important ``endosplit'' property. As a result, we can invoke the chain-complex theoretic Green correspondence.

    For this section, fix $S \in \Syl_p(G)$, $V$ an absolutely $p$-divisible $p$-permutation $kG$-module, and as before, let $\calX_V \subset s_p(G)$ be the set of $p$-subgroups of $G$ for which $V(P) = 0$. We first deduce some easy properties.

    \begin{prop} \label{omnibusrestrictionprops}
        Let $H \leq G$ with $S \leq H$ and let $x$ denote either $e$ or $w$.
        \begin{enumerate}
            \item $C \in Ch^b({}_{kG}\triv)$ is weakly $V$-endotrivial (resp. $V$-endosplit-trivial) if and only if $\Res^G_H C$ is weakly $\Res^G_H V$-endotrivial (resp. $\Res^G_H V$-endosplit-trivial).
            \item $\Res^G_H: x\calE^V_k(G) \to x\calE^{\Res^G_H V}_k(H)$ has kernel contained in $\ker (h_x)$. If $N_G(S) \leq H \leq G$ and $x = w$ or $e$, then $\Res^G_H: x\calE^V_k(G) \to x\calE^V_k(H)$ is injective.
            \item Let $D \in x\calE_k^{\Res^G_H V}(H)$. If $[D] \in \im\big( \Res^G_H:  x\calE^V_k(G) \to x\calE^{\Res^G_H V}_k(H)\big)$, then $\Ind^G_H D$ has a $V$-endotrivial complex direct summand.
            \item For every $C \in x\calE_k^V(G)$ there exists some $D \in x\calE_k^{\Res^G_H V}(H)$ such that $C \mid \Ind^G_H D$.
            \item For all $[D] \in \im \big(\Res^G_H: x\calE^V_k(G) \to x\calE^{\Res^G_H V}_k(H)\big)$, all $P \in s_p(G)$ satisfying $V(P) = 0$, and all $g \in G$ satisfying ${}^gP \leq H$, we have $h_D(P) = h_D({}^gP)$.
        \end{enumerate}
    \end{prop}
    \begin{proof}
        \begin{enumerate}
            \item For weak endotriviality, the forward direction is clear since restriction preserves all h-marks and $V$-endotriviality. Suppose $\Res^G_H C$ is $\Res^G_H V$ endotrivial. Then given any $P\leq S$ for which $(\Res^G_H V)(P) = 0$ (equivalently, $V(P) = 0$), $\Inf^{N_H(P)}_{N_H(P)/P}((\Res^G_H C)(P)) \cong \Res^{N_G(P)}_{N_H(P)} (C(P))$, regarding the Brauer construction as a functor ${}_{kG}\catmod \to {}_{kN_G(P)}\catmod$. Since $\Res^G_H C$ is $\Res^G_H V$-endotrivial, $(\Res^G_H C)(P) $ has nonzero homology in one degree with it having $k$-dimension one. Therefore $\Res^{N_G(P)}_{N_H(P)} (C(P))$ has this property, and $C(P)$ does too. Since all $p$-subgroups of $G$ are conjugate to a subgroup of $S$, we conclude that $C$ is endotrivial. The $V$-endosplit-trivial case follows by a similar argument.

            \item $\Res^G_H$ preserves the h-marks of $C$. Therefore, if $C \in \ker (\Res^G_H)$, all the h-marks of $C$ are 0, and therefore $\ker( \Res^G_H) \subseteq \ker (h_x)$. The latter claim follows from the Green correspondence for chain complexes, as in this case, $\Res^G_H$ induces an injective map from the set of isomorphism classes of trivial source $V$-endotrivial $kG$-modules to the set of isomorphism classes of trivial source $\Res^G_H V$-endotrivial $kH$-modules.

            \item By relative projectivity, $C \mid \Ind^G_H \Res^G_H C$. Therefore, if $D = \Res^G_H C$, $\Ind^G_H D$ has a $V$-endotrivial direct summand, namely $C$.

            \item Set $D = \Res^G_H C$. Then $C \mid \Ind^G_H \Res^G_H C$.

            \item This follows since ${}^gC(P) \cong C({}^gP)$, therefore h-marks of $V$-endotrivial $kG$-complexes are invariant under $G$-conjugation.
        \end{enumerate}
    \end{proof}

    \begin{definition}
        Let $S\leq H\leq G$, $W$ an absolutely $p$-divisible $p$-permutation $kH$-module, and $C$ be a $W$-endotrivial $kH$-complex. We say \textit{$C$ is $G$-stable} if for all pairs of $p$-subgroups $P,Q \in s_p(H)$ which are $G$-conjugate and such that $h_C(P), h_C(Q)$ are defined, $h_C(P)= h_C(Q)$. Define the \textit{$G$-stable subgroup} $x\calE^W_k(H)^G\leq x\calE^W_k(H)$ to be the subgroup consisting of all classes of $G$-stable $W$-endotrivial complexes.

        Similarly, let $H\leq G$ be any subgroup of $G$, and let $C$ be an endosplit $p$-permutation resolution of $kH$-modules. Say \textit{$C$ is $G$-stable} if for all $G$-conjugate pairs of $p$-subgroups $P, Q \in s_p(H)$ for which $C(P)$ and $C(Q)$ are noncontractible, $C(P)$ and $C(Q)$ have nonzero homology concentrated in the same unique degree.
    \end{definition}

    Since h-marks are preserved under conjugacy, the image of $\Res^G_H: \calE^V_k(G) \to \calE^{\Res^G_H V}_k(H)$ is contained in $\calE^V_k(H)^G$. Therefore, one interesting question is to determine when $\Res^G_H: \calE^V_k(G) \to \calE^{\Res^G_H V}_k(H)^G$ is surjective. To investigate further, we must describe how additive induction and the Brauer construction behave together. Note that as a functor, the Brauer construction sends a homomorphism to its corresponding quotient homomorphism after restriction to fixed points. 

    \begin{theorem}\label{inductionbrauercommute}
        Let $P\in s_p(G)$, $H \leq G$, and $M$ be a $p$-permutation $kH$-module. We have the following natural isomorphism of $kN_G(P)$-modules, regarding the Brauer construction as a functor $-(P): {}_{kG}\triv \to {}_{kN_G(P)}\triv$: \[\left(\Ind^G_H M\right)(P) \cong \bigoplus_{x \in [N_G(P)\backslash G / H], P\leq {}^xH} \Ind^{N_G(P)}_{N_G(P)\cap {}^xH}\big( ({}^xM)(P) \big).\]
        In particular, since the isomorphism is natural, it also holds in the chain complex category $Ch^b({}_{kN_G(P)}\triv)$ (see for instance \cite[Definition 5.4.10]{L181}).
    \end{theorem}

    \begin{proof}
        First, we have $(\Ind^G_H M)(P) = (\Res^G_{N_G(P)} \Ind^G_H M)(P)$. Applying the Mackey formula yields the following natural isomorphism of $kN_G(P)$-modules:

        \begin{align*}
            (\Res^G_{N_G(P)} \Ind^G_H M)(P) &\cong \left(\bigoplus_{x \in [N_G(P)\backslash G / H]} \Ind^{N_G(P)}_{N_G(P) \cap {}^xH} \Res^{{}^xH}_{N_G(P) \cap {}^xH} ({}^xM)\right)(P)\\
            &\cong \bigoplus_{x \in [N_G(P)\backslash G / H]} \left(\Ind^{N_G(P)}_{N_G(P) \cap {}^xH} \Res^{{}^xH}_{N_G(P) \cap {}^xH}({}^xM)\right)(P)
        \end{align*}
        Given a double coset representative $x$, each indecomposable constituent of $\Ind^{N_G(P)}_{N_G(P) \cap {}^xH} \Res^{{}^xH}_{N_G(P) \cap {}^xH} ({}^xM)$ has a vertex contained in $N_G(P)\cap {}^xH$, so if $\left(\Ind^{N_G(P)}_{N_G(P) \cap {}^xH} \Res^{{}^xH}_{N_G(P) \cap {}^xH} ({}^xM)\right)(P) \neq 0$, then $N_G(P) \cap {}^xH$ contains an $N_G(P)$-conjugate of $P$, and since $P \trianglelefteq N_G(P)$, this occurs if and only if $P \leq {}^xH$. Therefore, all terms in the direct sum that do not satisfy $P \leq {}^xH$ are zero, and we have
        \[(\Ind^G_H M)(P) \cong \bigoplus_{x \in [N_G(P)\backslash G / H], P\leq {}^xH} \left(\Ind^{N_G(P)}_{N_G(P) \cap {}^xH} \Res^{{}^xH}_{N_G(P) \cap {}^xH}({}^xM)\right)(P).\]
        We now compute the term inside the direct sum when $P \leq {}^xH$. First, since $P \trianglelefteq N_G(P)$,
        \[ \Ind^{N_G(P)}_{N_G(P) \cap {}^xH} ({}^xM)^P = \left(\Ind^{N_G(P)}_{N_G(P) \cap {}^xH} \Res^{{}^xH}_{N_G(P) \cap {}^xH}{}^xM\right)^P.\]
        Indeed, for any $g \in N_G(P)$, and $m \in {}^xM$, $p\cdot (g\otimes m) = g\otimes m $ if and only if $g \otimes p^g m = g\otimes m$ for all $p \in P$, and the equality follows from there.
        We next claim \[\sum_{Q < P}\tr^P_Q \left(\Ind^{N_G(P)}_{N_G(P) \cap {}^xH} \Res^{{}^xH}_{N_G(P) \cap {}^xH} {}^xM\right)^Q = \Ind^{N_G(P)}_{N_G(P) \cap {}^xH} \left(\sum_{Q < P} \tr^P_Q ({}^xM)^Q\right).\]
        Indeed, for any $Q < P$, $g \otimes m \in  \left(\Ind^{N_G(P)}_{N_G(P) \cap {}^xH} \Res^{{}^xH}_{N_G(P) \cap {}^xH} {}^xM\right)^Q$ if and only if $m \in (^xM)^{Q^g}$. Now given coset representatives $u_1, \dots, u_k$ of $P/Q$, $(u_1 + \dots + u_k)g \otimes m = g \otimes (u_1^g + \dots + u_k^g)m$, and since $u_1^g, \dots,  u_k^g$ are coset representatives of $P/Q^g$, we have an equality of sets. Finally, for arbitrary groups $H\leq G$ and appropriate modules $N \subseteq M$, $\Ind^G_H M/N \cong \Ind^G_H M / \Ind^G_H N$ naturally by exactness of additive induction. We now compute, assuming $P \leq {}^xH$ and setting $N_{{}^xH}(P) = N_G(P)\cap {}^xH$:
        \begin{align*}
            \left(\Ind^{N_G(P)}_{N_{{}^xH}(P)} \Res^{{}^xH}_{N_{{}^xH}(P)}({}^xM)\right)(P) &:= \left(\Ind^{N_G(P)}_{N_{{}^xH}(P)} \Res^{{}^xH}_{N_{{}^xH}(P)}{}^xM\right)^P / \sum_{Q < P}\tr^P_Q \left(\Ind^{N_G(P)}_{N_{{}^xH}(P)} \Res^{{}^xH}_{N_{{}^xH}(P)} {}^xM\right)^Q \\
            &= \Ind^{N_G(P)}_{N_{{}^xH}(P)} ({}^xM)^P/\Ind^{N_G(P)}_{N_{{}^xH}(P)} \left(\sum_{Q < P} \tr^P_Q ({}^xM)^Q\right)\\
            &\cong \Ind^{N_G(P)}_{N_{{}^xH(P)}}\left(({}^xM)^P/ \left(\sum_{Q < P} \tr^P_Q ({}^xM)^Q\right) \right)\\
            &= \Ind^{N_G(P)}_{N_{{}^xH(P)}}\big(({}^xM)(P)\big)
        \end{align*}
        This completes the proof of the isomorphism for modules, and since all the isomorphisms used were natural, the result follows. 
    \end{proof}

    As a first result of Theorem \ref{inductionbrauercommute}, we can generate endosplit $p$-permutation resolutions for summands of induced modules which have $G$-stable endosplit $p$-permutation resolutions.

    \begin{corollary}\label{fusionstableinduction}
        Let $H \leq G$ and let $W$ be an absolutely $p$-divisible $p$-permutation $kH$-module.
        \begin{enumerate}
            \item Suppose $S \leq H \leq G$ and let $C$ be a $G$-stable weakly $W$-endotrivial complex. Then for all $P \in \calX_V$ , $\left(\Ind^G_H C\right)(P)$ has homology concentrated in degree $h_C({}^gP)$, where $g \in G$ is chosen such that ${}^gP \leq S$.
            \item Suppose $S \leq H \leq G$ and let $C$ be a $G$-stable $W$-endosplit-trivial complex of $kH$-modules. Then for all $P \in s_p(G)$, $\left(\Ind^G_H C\right)(P)$ has homology concentrated in the degree $h_C({}^gP)$, where $g \in G$ is chosen such that ${}^gP \leq S$.
            \item Let $N$ be a $kH$-module with an endosplit $p$-permutation resolution $C$. $\Ind^G_H C$ is an endosplit $p$-permutation resolution if and only if $C$ is $G$-stable. If this occurs, $\Ind^G_H C$ is an endosplit $p$-permutation resolution for $\Ind^G_H N$.
        \end{enumerate}

    \end{corollary}

    \begin{proof}
        (a) and (b) follow similarly. Since the isomorphism from the previous theorem is natural, it extends to an isomorphism of chain complexes \[\left(\Ind^G_H(C)\right)(P) \cong \bigoplus_{x \in [N_G(P)\backslash G / H], P \leq {}^xH} \Ind^{N_G(P)}_{N_G(P)\cap {}^xH} \left( {}^xC(P) \right).\] Since induction is exact and additive, and taking homology is additive, for any $i \in \Z$, \[H_i\left(\left(\Ind^G_H(C)\right)(P)\right) \cong \bigoplus_{x \in [N_G(P)\backslash G / H], P \leq {}^xH} \Ind^{N_G(P)}_{N_G(P)\cap {}^xH} \left( H_i({}^xC(P)) \right).\]

        With $P \in s_p(G)$ or $\calX_V$ (corresponding to (a) or (b)), $h_C(P) = h_{^xC}(P) = h_C({}^xP) $ for all $x \in G$, by the assumption of $G$-stability. Therefore, the right-hand side of the equation implies the nonzero homology of the will fall in exactly one nonzero degree, which is $h_C({}^gP)$ for any $g \in G$ satisfying ${}^gP \leq S.$ Thus (a) and (b) hold.

        For (c), first assume $C$ is $G$-stable. By a similar computation as before, for every $P \in s_p(G)$, $(\Ind^G_H C)(P)$ either has homology concentrated in one degree or is contractible. In particular, $(\Ind^G_H C)(P)$ is contractible whenever $P$ is not conjugate to a subgroup of $H$. By Proposition \ref{prop:equivendosplit}, $\Ind^G_H C$ is an endosplit $p$-permutation resolution, and in this case $H_0(\Ind^G_H C) \cong \Ind^G_H N$.

        Conversely, if $C$ is not $G$-stable, then there exists an $x \in G$ and $p$-subgroup $P \leq H$ for which $C(P)$ and $C(^xP)$ have homology concentrated in different degrees. In this case, the previous theorem implies that $\left(\Ind^G_H C\right)(P)$ has nonzero homology in both degrees, and therefore $\Ind^G_H C$ cannot be an endosplit $p$-permutation resolution.
    \end{proof}

    We next prove a converse of Proposition \ref{omnibusrestrictionprops}(c).

    \begin{corollary}\label{inrestrictionimage}
        Let $x = e$ or $w$, and set $D \in x\calE_k^{\Res^G_S V}(S)$. $[D] \in \im\big(\Res^G_S: x\calE_k^V(G) \to x\calE_k^{\Res^G_S V}(S)\big)$ if and only if $D$ is $G$-stable and $\Ind^G_S D$ has a weakly $V$-endotrivial (resp. $V$-endosplit-trivial) complex as a direct summand.
    \end{corollary}
    \begin{proof}
        The forward direction is Proposition \ref{omnibusrestrictionprops}, (c). Suppose $\Ind^G_S D$ has a weakly $V$-endotrivial (resp. $V$-endosplit-trivial) complex as a direct summand $C$ and is $G$-stable. Since $D$ is $G$-stable, Theorem \ref{inductionbrauercommute} implies the h-marks at subgroups of $S$ for $C$ and $D$ coincide. Therefore, the h-marks of $\Res^G_S C \otimes_k D^*$ are all zero. Since $S$ is a $p$-group, $\ker (h) = \{[k]\}$ so $[\Res^G_S C] = [D] \in x\calE_k^{\Res^G_S V}(H)$, as desired.
    \end{proof}

    For (non-relatively) endotrivial complexes, the image of restriction to a Sylow $p$-subgroup can be completely described. Therefore, the question of classifying all endotrivial complexes reduces to classifying all endotrivial complexes for $p$-groups.

    \begin{theorem} \label{endotrivialrestriction}
        Let $S \in \Syl_p(G)$. Then the group homomorphism $\Res^G_S: \calE_{k}(G) \to \calE_k(S)^G$ is surjective. Moreover, we have a split exact sequence of abelian groups \[0 \to \Hom(G.k^\times) \to \calE_k(G) \xrightarrow{\Res^G_S} \calE_k(S)^G \to 0.\]
    \end{theorem}
    \begin{proof}
        Let $D \in \calE_k(S)^G$ and assume without loss of generality that $D$ is indecomposable. We have $H_{h_1(D)}(D) \cong k$. Since $D$ is an endosplit $p$-permutation resolution for $k$ as $kS$-module and is $G$-stable, $\Ind^G_S D$ is an endosplit $p$-permutation resolution for $\Ind^G_S H(C) \cong k[G/S]$, and moreover, for all $P \in s_p(S)$, $h_C(P)$ is the unique value of $i$ for which $H_i((\Ind^G_S C)(P)) \neq 0$. $k$ is a direct summand of $k[G/S]$, therefore there is a direct summand $C \mid \Ind^G_H D$ which is an endosplit $p$-permutation resolution for $k$ as $kG$-module, and whose h-marks coincide with those of $D$. Therefore, $\Res^G_S C$ has an indecomposable direct summand isomorphic to $D$.

        For the last statement, $\ker(\Res^G_S)$ consists of all endotrivial complexes with h-marks all zero. However, this is simply $\ker(h)$, which is the torsion subgroup of $\calE_k(G)$, isomorphic to $\Hom(G,k^\times)$. Therefore, the inclusion $\Hom(G,k^\times) \to \calE_k(G)$ splits with retraction $\calE_k(G) \to \Hom(G,k^\times)$ given by $[C] \mapsto H_{h_C(1)}(C) \in \Hom(G,k^\times)$.
    \end{proof}

    \subsection{Restriction for $V$-endosplit-trivial complexes}

    We next consider the case of restriction for $V$-endosplit-trivial complexes, and show that essentially, the question of restriction reduces to the question of restriction for relatively endotrivial modules. The following definition extends the usual notion of being a $G$-stable endopermutation module (see \cite[Section 9.9]{L182}).

    \begin{definition}
        Let $S \in \Syl_p(G)$. For any $H \geq S$, we say a $\Res^G_H V$-endotrivial $kH$-module is $G$-stable if for all $x \in G$, \[[\Res^H_{H\cap {}^xH} M] = [(\Res^{^xH}_{H\cap {}^xH} \circ c_x) (M)] \in T_{k, \Res^G_{H\cap{}^xH}V}(H\cap {}^xH)\] where $c_x$ denotes conjugation by $x$. Denote the subgroup of $T_{k, \Res^G_H V}(H)$ consisting of $G$-stable elements by $T_{k, \Res^G_H V}(H)^{G-st}$.

        If $H \leq G$, then $N_G(H)/H$ has a natural action on $T_{k, \Res^G_H V}(H)$ via conjugation. It is a straightforward verification that $T_{k, \Res^G_S V}(S)^{N_G(H)} = T_{k, \Res^G_S V}(S)^{N_G(H)-st}$, therefore \[T_{k, \Res^G_S V}(S)^{G-st} \leq T_{k, \Res^G_S V}(S)^{N_G(H)-st} = T_{k, \Res^G_S V}(S)^{N_G(H)}.\]

        Since for any $H \leq G$, $x \in G$, and $kG$-module $M$, \[(\Res^H_{H\cap {}^xH} \circ \Res^G_H) (M) = (\Res^{^xH}_{H\cap {}^xH} \circ c_x \circ \Res^G_H) (M),\]
        the image of $\Res^G_H: T_{k,V}(G) \to T_{k, \Res^G_H V}(H)$ is contained in $T_{k, \Res^G_H(V)}(H)^{G-st}$ when $H \geq S \in \Syl_p(G)$.

        Finally, we denote the subgroup of $T_{k,V}(G)$ consisting of endotrivial modules that have endosplit $p$-permutation resolutions by $eT_{k,V}(G)$. Recall that any $V$-endotrivial module belonging to $eT_{k,V}(G)$ must be endo-$p$-permutation.
    \end{definition}

    We break up our restriction into two parts: restriction to $N_G(S)$ and restriction to $S$. Recall that from Proposition \ref{omnibusrestrictionprops} (b), if $N_G(S) \leq H\leq G$, then $\Res^G_H: e\calE_k^V(G) \to e\calE_k^{\Res^G_H V}(H)$ is injective.

    \begin{theorem}\label{endosplitrestrictiontonormalizer}
        Let $N_G(S) \leq H \leq G$. Let $N$ be an indecomposable $\Res^G_H V$-endotrivial $kH$-module with indecomposable $G$-stable endosplit $p$-permutation resolution $D$. Then $[D] \in e\calE_k^{\Res^G_H V}(H)^G$ is in $\im\big(\Res^G_H: e\calE_k^{V}(G) \to e\calE_k^{\Res^G_H V}(H)^{G}\big)$ if and only if the Green correspondent of $N$ is $V$-endotrivial. If this occurs, the Green correspondent of $D$ is an endosplit $p$-permutation resolution of the Green correspondent of $N$, and has the same h-marks.

        In particular, if the Green correspondence induces an isomorphism $eT_{k,V}(G) \cong eT_{k, \Res^G_H V}(H)^{G-st}$, it induces an isomorphism $e\calE_k^V(G) \to e\calE_k^{\Res^G_H V}(H)^G$. If $H$ controls fusion in $G$, the converse holds.

    \end{theorem}
    \begin{proof}
        First, suppose the Green correspondent of $N$ is $V$-endotrivial, and denote it by $M$. Since $M \mid \Ind^G_H N$, and $\Ind^G_H D$ is an endosplit $p$-permutation resolution for $\Ind^G_H M$, there exists a unique direct summand of $\Ind^G_H D$ which is an indecomposable endosplit $p$-permutation resolution for $M$, which we denote by $C$. Moreover, since $M$ has vertex $S$, it follows that its indecomposable endosplit $p$-permutation resolution must have vertex $S$ as chain complex. Hence, $C$ is the indecomposable endosplit $p$-permutation resolution of $M$, and it follows by the Green correspondence that $[\Res^G_H C] = [D] \in e\calE_k^{\Res^G_H V}(H)$.

        Now, suppose $[D] \in \im\big(\Res^G_H: e\calE_k^{V}(G) \to e\calE_k^{\Res^G_H V}(H)^{G}\big)$. Therefore, $\Ind^G_H D$ has a $V$-endotrivial complex as a direct summand. Since $V$-endotrivial complexes have vertex $S$, it follows that the Green correspondent of $D$ is the $V$-endotrivial direct summand. Call the Green correspondent $C$, then since $D$ is an endosplit $p$-permutation resolution, $\Ind^G_H D$ is an endosplit $p$-permutation resolution for $\Ind^G_H N$, and $C$ is an endosplit $p$-permutation resolution for some summand of $\Ind^G_H N$ by Proposition \ref{omnibusrestrictionprops}. We claim that $C$ is an endosplit $p$-permutation resolution for the Green correspondent of $N$, which we denote by $M$. Indeed, $M$ is a $V$-endotrivial $kG$-module, hence it has vertex $S$, and it has an endosplit $p$-permutation resolution since it a summand of $\Ind^G_H N$. Therefore, its indecomposable endosplit $p$-permutation resolution must have vertex $S$ as well, so the resolution must be $C$.

        The forward implication of the final statement follows from the rest of the theorem, noting that $\Res^G_H$ is injective by the Green correspondence. If $H$ controls the fusion of $S$ in $G$, every endosplit $p$-permutation resolution of a $V$-endotrivial $kH$-module is $G$-stable, and the converse follows.
    \end{proof}

    The main obstruction to the converse statement holding in general is the following. It is unknown if every element of $eT_{k, \Res^G_H V}(H)^{G-st}$ has an endosplit $p$-permutation resolution which is $G$-stable.

    \begin{corollary}\label{endotrivialrestrictiontonormalizer}
        Let $N_G(S) \leq H \leq G$. $\Res^G_H \calE_k(G) \to \calE_k(H)$ is injective. It is an isomorphism if and only if $\Res^G_H: \Hom(G,k^\times) \to \Hom(H, k^\times)$ is an isomorphism, or equivalently, if the Green correspondent of any $k$-dimension one $kH$-module has $k$-dimension one.
    \end{corollary}
    \begin{proof}
        The first statement follows from Proposition \ref{omnibusrestrictionprops} (b). The latter follows from the identification $T_{k, \{0\}}(G) = \Hom(G,k^\times)$ and Theorem \ref{endosplitrestrictiontonormalizer}.
    \end{proof}

    We next examine the case in which $S \trianglelefteq P$, and consider restriction to $S$. In the non-relative endotrivial module setting, there is an isomorphism $TF(S)^{G/S} \cong TF(G)$, where $TF$ denotes the torsion-free part of $T(G)$, originally shown by Dade and further studied by Mazza in \cite{Ma07}. The original statement, given in an unpublished manuscript of Dade, is as follows.

    \begin{theorem}{\cite[Theorem 7.1]{Da82}}\label{dadethm}
        Let $S\in \Syl_p(G)$ with $S \trianglelefteq G$ and let $M$ be an endopermutation $kS$-module. Then $M$ extends to an endopermutation $kG$-module $N$ (i.e. there exists an endopermutation module $N$ with $\Res^G_S N \cong M$) if and only if $M$ is $G$-stable, i.e. ${}^gM \cong M$ for all $g \in G$.
    \end{theorem}

    \begin{remark}
        Let $V$ be an absolutely $p$-divisible $p$-permutation $kG$-module, and $M$ a $V$-endotrivial $kG$-module with an endosplit $p$-permutation resolution $C$. Then $C^* \otimes_k C \simeq (M^* \otimes_k M)[0] \cong (k \oplus P)[0]$, where $P$ is a $V$-projective module. Moreover, since $C$ is a chain complex of $p$-permutation modules, $P$ is $p$-permutation. Therefore, $M$ is an endo-$p$-permutation module, and $\Res^G_S M$ is an endopermutation module.

        Conversely, suppose $N$ is a $G$-stable $\Res^G_S V$-endotrivial $kS$-module which is endopermutation. By Dade's lifting theorem  \ref{dadethm}, there exists a $kG$-module $M$ for which $\Res^G_S M \cong N$. By \cite[Proposition 4.1.2]{CL11}, since $S \trianglelefteq G$, $M$ is $V$-endotrivial if and only if $\Res^G_S M$ is $\Res^G_S V$-endotrivial, so $M$ is $V$-endotrivial as well.

        Putting this together, we obtain the following results for endo-$p$-permutation $V$-endotrivial modules and $V$-endosplit-trivial complexes.
    \end{remark}

    \begin{prop}
        Suppose $S \in \Syl_p(G)$ is normal in $G$. Then $\Res^G_S: eT_{k,V}(G)\to eT_{k, \Res^G_S V}(S)^{G-st}$ is surjective.
    \end{prop}
    \begin{proof}
        This follows from the arguments in the previous remark, as every endo-$p$-permutation $kS$-module is endopermutation.
    \end{proof}

    \begin{theorem}\label{endosplitnormalrestriction}
        Suppose $S \in \Syl_p(G)$ is normal in $G$. Then $\Res^G_S: e\calE_k^V(G) \to e\calE_k^{\Res^G_S V}(G)^G$ is surjective.
    \end{theorem}
    \begin{proof}
        Suppose $C$ is a $G$-stable $\Res^G_S V$-endosplit-trivial complex of $kS$-modules, and without loss of generality assume $C$ is indecomposable. It suffices by Theorem \ref{inrestrictionimage} to show that $\Ind^G_S C$ has a $V$-endosplit-trivial summand. Set $H_{h(C)}(C) \cong N$, then $N$ is an indecomposable $\Res^G_S V$-endotrivial $kS$-module.

        By Theorem \ref{fusionstableinduction}, $\Ind^G_S C$ is an endosplit $p$-permutation resolution for $\Ind^G_S N$. Furthermore, from the above remark, $N$ has a $V$-endotrivial $kG$-module lift $M$, so we may write $N \cong \Res^G_S M$. By relative projectivity, $M \mid \Ind^G_S\Res^G_S M \cong \Ind^G_S N$, and by properties of endosplit $p$-permutation resolutions, there is a corresponding direct summand of $\Ind^G_S C$ which is an endosplit $p$-permutation resolution of $M$, which is necessarily $V$-endosplit-trivial, as desired.
    \end{proof}

    \subsection{Restriction for weakly $V$-endotrivial complexes}

    We study the more general case of restriction for weakly $V$-endotrivial complexes. In this case, we still have the Green correspondence but have less control over its behavior.

    \begin{theorem}
        Let $S \in \Syl_p(G),$ $N_G(S) \leq L \leq G$, and let $V$ be an absolutely $p$-divisible $p$-permutation $kG$-module. Suppose $D$ is a $G$-stable indecomposable weakly $V$-endotrivial $kL$-complex, and let $C$ denote its Green correspondent. If for all $P \in \calX_V$, $C(P)$ has exactly one noncontractible summand, then $[D]\in \im\big (\Res^G_L: w\calE^V_k(G) \to w\calE^{\Res^G_L V}_k(L)^G\big)$.
    \end{theorem}

    \begin{proof}
         It suffices to show that $C$ is a weakly $V$-endotrivial complex, since from this it will follow by the Green correspondence that $[D] = [\Res^G_L C]$. It suffices to show for all $P \in \calX_V$ that $C(P)$ is an endotrivial $k[N_G(P)/P]$-complex. It follows from Corollary \ref{fusionstableinduction} that $(\Ind^G_L D)(P)$ is an endosplit $p$-permutation resolution, therefore any noncontractible summand is an endosplit $p$-permutation resolution as well. In particular, since $C(P)$ is noncontractible, it is an endosplit $p$-permutation resolution. It suffices to show that $C(P)$ has non-zero homology with $k$-dimension one.

        The nonzero homology of $(\Ind^G_L D)(P)$ is given by \[\calH\big((\Ind^G_L D)(P)\big) = \bigoplus_{x \in [N_G(P)\backslash G / L], P \leq {}^xL} \Ind^{N_G(P)}_{N_G(P)\cap {}^xL}(\calH({}^xD(P))),\] where $\calH$ denotes the degree in which the nonzero homology occurs. Note that $\dim_k \calH({}^xD(P)) = 1$ since $D$ is weakly $\Res^G_L V$-endotrivial.

        There exists some $x \in G$ for which $p \nmid [N_G(P) : N_G(P) \cap {}^xL]$. Indeed, if $Q \in \Syl_p(N_G(P))$, there exists $x \in G$ such that $Q \leq {}^xS$ and $S \leq N_G(S) \leq L$. Therefore, $\calH\big((\Ind^G_L D)(P)\big)$ contains a 1-dimensional $k[N_G(P)/P]$-module, say $\chi$, as a direct summand. Since $(\Ind^G_L D)(P)$ is an endosplit $p$-permutation resolution, there exists a unique corresponding indecomposable endosplit $p$-permutation resolution for $\chi$ which is a direct summand of $(\Ind^G_L D)(P)$, i.e. an indecomposable endotrivial complex. Suppose $E'$ is the unique endotrivial $k[N_G(P)/P]$-complex which is a direct summand of $(\Ind^G_L D)(P)$. Then, there exists a corresponding unique indecomposable direct summand $E$ of $\Ind^G_L D$ such that $E'$ is a summand of $E(P)$.

        Now, we claim $E = C$. It suffices to show that $E$ has vertex $S$. We have that \[(E^*\otimes_k E)(P)\cong E(P)^* \otimes_k E(P) \simeq k[0].\] Therefore by the previous lemma, $E^*\otimes_k E$ is homotopy equivalent to a complex $F$ with $F_0$ containing a direct summand with vertex $S$. It follows that $E$ must also have vertex $S$, but since there is a unique summand of $\Ind^G_L D$ with vertex $S$ by the Green correspondence, $E = C$.

        By hypothesis, the only noncontractible summand of $C(P)$ is $E'$. Therefore, $C(P)$ is an endotrivial $k[N_G(P)/P]$-complex for all $P \in \calX_V$, and we conclude that $C$ is weakly $V$-endotrivial.

    \end{proof}

    We can also extend an argument of Lassueur in \cite{CL11} which determines that the Green correspondence induces an isomorphism when ${}_{kG}\catmod(V)$ is sufficiently large.

    \begin{theorem}\label{sufficientlylargeX}
        Let $N_G(S)\leq H \leq G$ and suppose ${}_{kH}\catmod(\Res^G_H V) \supseteq {}_{kH}\catmod(\mathfrak{X})$, where $\mathfrak{X}$ is the set of subgroups from \cite[Theorem 5.2.1(ii)]{L181},
        \[\mathfrak{X} = \{gSg\inv \cap S \mid g \in G\setminus H\}.\]
        Then $\Res^G_H: x\calE^V_k(G) \to x\calE^{\Res^G_H V}_k(H)^G$ is an isomorphism.
    \end{theorem}
    \begin{proof}
        $\Res^G_H $ has been shown to be injective already, so it suffices to show it is surjective (in fact, it remains to show only for $x = w$). Let $D$ be an indecomposable, $G$-stable, $\Res^G_H V$-endotrivial $kH$-complex. We have \[\Res^G_H \Ind^G_H D \cong \bigoplus_{x \in [H\backslash G / H]} \Ind^H_{H\cap {}^xH}\Res^H_{H\cap {}^xH} D.\] From the proof of the Green correspondence (see \cite[Theorem 5.2.1]{L181}]), it follows that exactly one of these summands has vertex $S$, and the rest have vertex contained within $S \cap {}^x S$ for some $x \not\in N_G(S)$. Thus, for each complex with vertex not conjugate to $S$, each of its components is $\mathfrak{X}$-projective, hence $\Res^G_H V$-projective. Thus, $\Res^G_H \Ind^G_H [D] = [D] \in x\calE_k^{\Res^G_H V}(H)^G$, as desired.
    \end{proof}

    \begin{remark}
        One important case where the hypothesis for Theorem \ref{sufficientlylargeX} is satisfied is when $V = V(\calF_G)$ is the following module defined by Lassueur in \cite{CL13}.
        \[V(\calF_G) := \bigoplus_{P \in s_p(G), P \not\in \Syl_p(G)} k[G/P].\] A subgroup of the group $T_{V(\calF_G)}(G)$ generalizes the Dade group for arbitrary groups, and coincides with the generalization of the Dade group for fusion systems introduced by Linckelmann and Mazza in \cite{LiMa09}.
    \end{remark}

    \textbf{Acknowledgements:} The author would like to thank Nadia Mazza and Caroline Lassueur for suggesting that he think about endotrivial complexes in a relative projective setting. He would also like to thank his advisor, Robert Boltje, for many hours of mentorship, thoughtful discussion, and feedback. Finally, he would like to thank the referee for their time and many helpful suggestions to improve the paper. 

    \bibliography{bib}

\begin{thebibliography}{10}

\bibitem{AC86}
M.~Auslander and J.~F. Carlson.
\newblock Almost-split sequences and group rings.
\newblock {\em J. Algebra}, 103(1):122--140, 1986.

\bibitem{Ba13}
P.~Balmer.
\newblock Modular representations of finite groups with trivial restriction to
  {Sylow} subgroups.
\newblock {\em J. Eur. Math Soc.}, 15(6):2061–2079, 2013.

\bibitem{BC86}
D.~J. Benson and J.~F. Carlson.
\newblock Nilpotent elements in the {Green} ring.
\newblock {\em J. Algebra}, 104(2):329--350, 1986.

\bibitem{BoYa06}
S.~Bouc and E.~Yal\c{c}in.
\newblock {Borel-Smith} functions and the {Dade} group.
\newblock {\em J. Algebra}, 311:821--839, 2007.

\bibitem{C96}
J.~F. Carlson.
\newblock {\em Modules and Group Algebras}.
\newblock Birkha\"user Basel, 1996.

\bibitem{CP96}
J.~F. Carlson and C.~Peng.
\newblock Relative projectivity and ideals in cohomology rings.
\newblock {\em J. Algebra}, 183(3):929--948, 1996.

\bibitem{CaTh00}
J.~F. Carlson and J.~Th\'evenaz.
\newblock Torsion endo-trivial modules.
\newblock {\em Algebras and Representation Theory}, 3:303–335, 2000.

\bibitem{CWZ20}
J.~F. Carlson, L.~Wang, and J.~Zhang.
\newblock Relatively projectivity and the {Green} correspondence for complexes.
\newblock {\em J. Algebra}, 560:879--913, 2020.

\bibitem{Da82}
E.~C. Dade.
\newblock Extending endo-permutation modules.
\newblock 1982, Unpublished manuscript.

\bibitem{Gr22}
J.~Grodal.
\newblock Endotrivial modules for finite groups via homotopy theory.
\newblock {\em J. Amer. Math Soc.}, 36(1):177–250, 2023.

\bibitem{HL00}
M.~E. Harris and M.~Linckelmann.
\newblock Splendid derived equivalences for blocks of finite $p$-solvable
  groups.
\newblock {\em J. Lon. Math Soc.}, 62(1):85--96, 2000.

\bibitem{CL11}
C.~Lassueur.
\newblock Relative projectivity and relative endotrivial modules.
\newblock {\em J. Algebra}, 337(1):285--317, 2011.

\bibitem{CL12}
C.~Lassueur.
\newblock {\em Relative projectivity and relative endotrivial modules}.
\newblock PhD thesis, Ecole Polytechnique Federale De Lausanne, 2012.

\bibitem{CL13}
C.~Lassueur.
\newblock The {Dade} group of a finite group.
\newblock {\em J. Pure Appl. Algebra}, 217(1):97--113, 2013.

\bibitem{L181}
M.~Linckelmann.
\newblock {\em The Block Theory of Finite Group Algebras, Volume 1}.
\newblock Cambridge University Press, Cambridge, 2018.

\bibitem{L182}
M.~Linckelmann.
\newblock {\em The Block Theory of Finite Group Algebras, Volume 2}.
\newblock Cambridge University Press, Cambridge, 2018.

\bibitem{LiMa09}
M.~Linckelmann and N.~Mazza.
\newblock The {Dade} group of a fusion system.
\newblock {\em J. Group Theory}, 12:55--74, 2009.

\bibitem{Ma03}
N.~Mazza.
\newblock {\em Modules d’endo-permutation}.
\newblock PhD thesis, Facult´e des Sciences de l’Universit´e de Lausanne,
  2003.

\bibitem{Ma07}
N.~Mazza.
\newblock The group of endotrivial modules in the normal case.
\newblock {\em J. Pure Appl. Algebra}, 209:311--323, 2007.

\bibitem{Ma19}
N.~Mazza.
\newblock {\em Endotrivial Modules}.
\newblock SpringerBriefs in Mathematics, 2019.

\bibitem{SKM23}
S.~K. Miller.
\newblock Endotrivial complexes.
\newblock {\em J. Algebra}, 650:173--218, 2024.

\bibitem{NT89}
H.~Nagao and Y.~Tsushima.
\newblock {\em Representations of Finite Groups}.
\newblock Academic Press, 1989.

\bibitem{O}
T.~Okuyama.
\newblock A generalization of projective covers of modules over finite group
  algebras.
\newblock Unpublished manuscript.

\bibitem{Pu90}
L.~Puig.
\newblock Affirmative answer to a question of {Feit}.
\newblock {\em J. Algebra}, 131(2):513--526, 1990.

\bibitem{R96}
J.~Rickard.
\newblock Splendid equivalences: Derived categories and permutation modules.
\newblock {\em Proceedings of the London Mathematical Society}, 72:331--358,
  1996.

\bibitem{RR01}
R.~Rouquier.
\newblock Block theory via stable and {Rickard} equivalences.
\newblock In {\em Modular Representation Theory of Finite Groups: Proceedings
  of a Symposium held at the University of Virginia, Charlottesville}. De
  Gruyter, 2001.

\bibitem{td87}
T.~tom Dieck.
\newblock {\em Transformation Groups}.
\newblock de Gruyter Stud. Math, 1987.

\bibitem{Ur06}
J.-M. Urfer.
\newblock {\em Modules d’endo-p-permutation}.
\newblock PhD thesis, EPFL, 2006.

\end{thebibliography}
    \bibliographystyle{plain}

\end{document}